\documentclass[a4paper]{amsart}
\pdfoutput=1
\usepackage[T1]{fontenc}

\usepackage{lmodern}

\usepackage{microtype}

\usepackage[pdftitle={Semiclassical stationary states for nonlinear Schr\"odinger equations under a strong external magnetic field}, pdfauthor={Jonathan Di Cosmo and Jean Van Schaftingen}]{hyperref}

\usepackage[abbrev]{amsrefs}

\usepackage[vmargin=4cm]{geometry}

\usepackage{enumerate}

\title[Nonlinear Schr\"odinger equation]{Semiclassical stationary states for nonlinear Schr\"odinger equations under a strong external magnetic field}

\author{Jonathan Di Cosmo}
\address{
  Institut de Recherche en Math\'ematique et Physique\\
  Universit{\'e} catholique de Louvain\\
  Chemin du Cyclotron 2 bte L7.01.01, 1348 Louvain-la-Neuve, Belgium}
\address{
  D{\'e}partement de Math{\'e}matique\\
  Universit{\'e} libre de Bruxelles, CP 214\\
  Boulevard du Triomphe, 1050 Bruxelles, Belgium}

\author{Jean Van Schaftingen}
\address{
  Universit{\'e} catholique de Louvain\\
  Institut de Recherche en Math\'ematique et Physique\\
  Chemin du Cyclotron 2 bte L7.01.01\\ 1348 Louvain-la-Neuve\\ Belgium}
\email{Jean.VanSchaftingen@uclouvain.be}

\newcommand{\R}{\mathbb{R}}
\newcommand{\N}{\mathbb{N}}
\newcommand{\C}{\mathbb{C}}
\newcommand{\dif}{\,\mathrm{d}}
\newcommand{\norm}[1]{\left\| #1 \right\|}		% norme
\newcommand{\abs}[1]{\lvert #1 \rvert}			% valeur absolue
\newcommand{\bigabs}[1]{\bigl\lvert #1 \bigr\rvert}			% valeur absolue
\newcommand{\Bigabs}[1]{\Bigl\lvert #1 \Bigr\rvert}			% valeur absolue
\DeclareMathOperator{\supp}{supp}		% support d'une fonction
\newcommand{\weakto}{\rightharpoonup}
\newcommand{\dualprod}[2]{\langle #1, #2 \rangle}
\newcommand{\scalprod}[2]{( #1 \vert #2 )}
\newcommand{\mathsuper}[1]{#1}%Superposition operators
\newcommand{\defeq}{\buildrel\triangle\over =}

\newcommand{\Lin}{\mathcal{L}}
\newcommand{\e}{\mathrm{e}}
\newcommand{\st}{\;\vert\;}

\DeclareMathOperator{\sign}{sign}

\newtheorem{proposition}{Proposition}[section]
\newtheorem{theorem}{Theorem}
\newtheorem{lemma}[proposition]{Lemma}

\keywords{Magnetic nonlinear Schr\"odinger equation; strong magnetic field; large magnetic field; magnetic dipole; Lorentz force; Nehari manifold; penalization method; asymptotics of solutions; estimates for distributional solutions; concentration.}
\subjclass[2010]{35J91 (35B25, 35B40, 35J20, 35Q55, 81Q20)}

\begin{document} 

\begin{abstract}
We construct solutions to the nonlinear magnetic Schr\"odinger equation 
\[
  \left\{ \begin{aligned} 
  - \varepsilon^2 \Delta_{A/\varepsilon^2} u + V u &= \abs{u}^{p-2} u & &\text{in}\  \Omega, \\
  u &= 0 & &\text{on}\  \partial\Omega,
  \end{aligned}
  \right.
\]
in the semiclassical r\'egime under strong magnetic fields.
In contrast with the well-studied mild magnetic field r\'egime, the limiting energy depends on the magnetic field allowing to recover the Lorentz force in the semi-classical limit.
Our solutions concentrate around global or local minima of a limiting energy that depends on the electric potential and on the magnetic field.
Our results cover unbounded domains, fast-decaying electric potential and unbounded electromagnetic fields.
The construction is variational and is based on an asymptotic analysis of solutions to a penalized problem following the strategy of M.\thinspace{}del Pino and P.\thinspace{}Felmer.
\end{abstract}

\maketitle

\section{Introduction}

The nonlinear Schr\"odinger equation that models the evolution of a wave-function \(\psi\) of a charged particle in \(\R^N\) under an external electromagnetic field reads as
\begin{equation}
\label{eqPhysicalMagNLSE}
  i \hbar \partial_t \psi = -\frac{\hbar^2}{2 m} \Delta_{A/\hbar} \psi + U \psi - \abs{\psi}^{p - 2} \psi. 
\end{equation}
Here \(m\) is the mass of the particle, \(2\pi \hbar = h\) is the Planck constant, the function \(U : \R^N \to \R\) is an external electric potential, the differential form \(A : \R^N \to \bigwedge^1 \R^N\) is an external magnetic vector potential, \(\Delta_A\) is the covariant Laplacian with respect to the connexion induced by \(A\) and given for \(\psi \in C^2 (\R^N; \C)\) by
\[
  -\Delta_A \psi \defeq -\Delta \psi - 2 \imath \scalprod{D\psi}{A} - \imath (d^* A) \psi + \abs{A}^2\psi,
\]
with \(d^* A (x) = \sum_{i = 1}^n D A (x)[e_i, e_i]\),
and \(-\abs{\psi}^{p - 2}\) is a focusing self-interaction potential. 
Stationary solutions to this problem have been studied in various settings \citelist{\cite{EstebanLions1989}\cite{ArioliSzulkin2003}\cite{ChabrowskiSzulkin2005}\cite{Chabrowski2002}}.

In the \emph{semi-classical limit}, that is when the scale of the problem is large compared to the Planck constant, one expects physically the motion to reduce to the classical Newtonian dynamics of a charged particle under the Lorentz force  
\begin{equation}
\label{eqLorentz}
  F = q (d U + v \lrcorner d A),
\end{equation}
where \(q\) is the electric charge and \(v\) the velocity vector of the charged particle. The inner product and the exterior derivative are given by \(v \lrcorner d A (x) [w] = d A (x)[v, w] = D A (x)[v, w] - D A (x) [w, v]\).
(In the three-dimensional Gibbs formalism, the formula \eqref{eqLorentz} corresponds to \(F = q (\nabla U + v \times \nabla \times A)\).)
In particular, standing wave solutions should correspond to particles at rest (\(v = 0\)) at critical points of the electric potential \(U\). 
In this situation the magnetic potential thus does not play any role. The 
corresponding stationary problem in the semi-classical limit has been the object 
of numerous studies in the last decade 
\citelist{\cite{Kurata2000}\cite{AlvesFigueiredoFurtado2011}\cite{Barile2008}
\cite{Cingolani2003}\cite{CingolaniClapp2010}\cite{CingolaniClapp2009}\cite{
CingolaniSecchi2002}\cite{CingolaniSecchi2005}\cite{DingLiu2013}\cite{
CingolaniJeanjeanSecchi2009}\cite{BartschDancerPeng2006}\cite{Barile2008b}\cite{
BarileCingolaniSecchi2006}\cite{CaoTang2006}\cite{SecchiSquassina2005}}.
(The complete Lorentz force \eqref{eqLorentz} can be recovered by studying  the 
soliton dynamics \cite{Squassina2009}.)

Since the standing waves in the r\'egime presented above do not interact with 
the magnetic field,  they do not allow to derive in the semi-classical limit the 
magnetic contribution to the Lorentz force.
Even if the Lorentz force does not act on charges at rest, it does act on 
\emph{magnetic 
dipoles} at rest according to the law
\begin{equation}
\label{eqLorentzDipole}
  F= q \, dU + d \dualprod{dA}{\mu},
\end{equation}
where the bivector \(\mu \in \bigwedge^2 \R^N\) is the magnetic moment of the dipole \cite{Kovanen2011}*{(1)}. (In the three-dimensional space in the Gibbs formalism, this is \(F = q \nabla U + \nabla (\mu \cdot \nabla \times A)\).)
Whereas the magnetic moment does not vanish in general, it does not play any role in the stationary semi-classical limit.
This can be explained as follows: if a wave-packet is concentrated at a length-scale \(\ell \approx \hbar/\sqrt{m E_0}\) (where \(E_0\) is the groundstate energy of the sytem), then the electric charge is of the order \(\ell^N\) whereas the magnetic dipole is at most of the order \(\ell^{N + 1}\). 
In order to study the interaction between the magnetic dipole and the magnetic field for stationary solutions, we propose to take an external magnetic potential of the order \(\ell^{-1}\), in a what we call the \emph{strong magnetic field r\'egime}. The interaction with the magnetic field should be comparable to the interaction with the electric field. This should allow to determine whether the classical Lorentz interaction of a charged magnetic dipole with an electromagnetic fied \eqref{eqLorentzDipole} is recovered in the stationary semi-classical limit.

By adimensionalization of the problem \eqref{eqPhysicalMagNLSE}, we are lead to study the mathematical problem
\begin{equation}
\label{problemMNLSE}\tag{$\mathcal{P}_{\varepsilon}$}
	\left\{ \begin{aligned} 
	- \varepsilon^2 \Delta_{A/\varepsilon^2} u + V u &= \abs{u}^{p-2} u & &\text{in}\  \Omega, \\
	u &= 0 & &\text{on}\  \partial\Omega,
       \end{aligned}
	\right.
\end{equation}
where \(\Omega \subset \R^N\) and \(\varepsilon > 0\) is a small real parameter.
The action functional associated to the problem \eqref{problemMNLSE} is
\begin{align*}
  \mathcal{F}_{\varepsilon}(u) \defeq  \frac{1}{2} \int_{\Omega} \varepsilon^2 \abs{D_{A/\varepsilon^2} u}^2 + V \abs{u}^2  
  - \frac{1}{p} \int_{\Omega} \abs{u}^{p};
\end{align*}
and is defined (with possibly the value \(-\infty\)) on the completion 
\(H^1_{V, A/\varepsilon^2} (\Omega)\) of the set of compactly supported continuously differentiable functions \(C^1_c (\Omega; \C)\) with respect to the Euclidean norm \(\norm{\cdot}\) defined for every \(u \in C^1_c (\Omega; \C)\) by
\[
  \norm{u} \defeq \Bigl(\int_{\Omega} \varepsilon^2 \abs{D_{A/\varepsilon^2} u}^2 + V \abs{u}^2\Bigr)^\frac{1}{2},
\]
where the covariant derivative \(D_{A/\varepsilon^2}\) is defined as
\[
 D_{A/\varepsilon^2} u \defeq Du + \imath u A/\varepsilon^2.
\]

In order to describe the limiting behaviour of solutions, given a real number \(V_* \in (0, \infty)\) and a form \(A_* \in \Lin (\R^N; \bigwedge^1 \R^N)\), we define 
% the \emph{magnetic Sobolev space \(H^1_{V_*, A_*} (\R^N)\)} by
% \[
%   H^1_{V_*, A_*} (\R^N) \defeq \bigl\{ v \in L^2 (\R^N; \C) \cap W^{1, 1}_\mathrm{loc} (\R^N; \C) \st D_{A_*} u \in L^2 (\R^N)\bigr\}
% \]
% and 
the \emph{limiting action functional} \(\mathcal{I}_{V_*, A_*} : H^1_{V_*, A_*} (\R^N) \to \R\) for every \(v \in H^1_{V_*, A_*} (\R^N)\) by 
\[
  \mathcal{I}_{V_*, A_*} (v) \defeq \frac{1}{2} \int_{\R^N} \abs{D_{A_*} u}^2 + V_* \abs{u}^2  
  - \frac{1}{p} \int_{\R^N} \abs{u}^{p}.
\]
We consider the \emph{action of the limiting problem}
\begin{equation}
\label{equationDefinitionE}
  \mathcal{E} (V_*, A_*)
  \defeq \inf \bigl\{ \mathcal{I}_{V_*, A_*} (v) \st v \in H^1_{V_*, A_*} (\R^N) \setminus \{0\} \text{ and }  \dualprod{\mathcal{I}_{V_*, A_*}'(v)}{v} = 0 \bigr\},
\end{equation}
that is, the infimum of the functional \(\mathcal{I}_{V_*, A_*}\) on its \emph{Nehari manifold}.
The infimum in the definition of the function \(\mathcal{E}\) is achieved at minimizers \cite{EstebanLions1989} that satisfy the limiting equation
\begin{equation}
\label{problemLimit}
\tag{$\mathcal{R}_{V_*, A_*}$}
  -\Delta_{A_*} v + V_* v = \abs{v}^{p - 2} v.
\end{equation}
The function \(\mathcal{E}\) is continuous (see proposition~\ref{propositionEcontinuous} below).

The \emph{concentration function} \(\mathcal{C} : \Omega \to \R\) is defined at every point \(x \in \Omega\) by 
\[
  \mathcal{C} (x) \defeq \mathcal{E} \bigl(V (x), D A (x)\bigr) = \mathcal{E} \bigl(V (x), d A (x)/2\bigr).
\]
The second equality comes from the gauge invariance of the limiting problem.
This function \(\mathcal{C}\) is continuous as soon as \(\frac{1}{2} - \frac{1}{N} < \frac{1}{p} < \frac{1}{2}\), \(V\) is continuous and \(A\) is continuously differentiable.

\begin{theorem}
\label{theoremGlobal}
Assume that \(\frac{1}{2} - \frac{1}{N} < \frac{1}{p} < \frac{1}{2}\), \(\Omega \subset \R^N\) is bounded and \(\inf V > 0\). If \(V \in C (\Bar{\Omega})\), \(A \in C^1 (\Bar{\Omega})\) and \(\Omega\) has a \(C^1\) boundary in a neighbourhood, then there exists a family of solutions \((u_\varepsilon)_{\varepsilon > 0}\) of \eqref{problemMNLSE} in \(H^1_{V,A/\varepsilon^2} (\Omega)\) and a family of points \((x_\varepsilon)_{\varepsilon > 0}\) in \(\Omega\) such that 
\begin{align*}
\lim_{\substack{R \to \infty\\ \varepsilon \to 0}} \norm{u_\varepsilon}_{L^\infty (\Omega \setminus B_\rho (x_\varepsilon))} &= 0,&
  & \lim_{\varepsilon \to 0} \varepsilon^{-N} \mathcal{F}_\varepsilon (u_\varepsilon)
  = \lim_{\varepsilon \to 0} \mathcal{C} (x_\varepsilon) = \inf_{\Omega} \mathcal{C}.  
\end{align*}
\end{theorem}

Since the concentration function \(\mathcal{C}\) is continuous and the set \(\Bar{\Omega}\) is compact, the function \(\mathcal{C}\) achieves its minimum on the set \(\Bar{\Omega}\). Theorem~\ref{theoremGlobal} can be rephrased by saying that the solution has a single spike that concentrates as \(\varepsilon \to 0\) towards the set of minimum points of the function \(\mathcal{C}\) on the set \(\Bar{\Omega}\).

In the particular case where the potential \(V\) is constant, the point of concentration is determined by the magnetic field \(d A\) alone. In particular, if the magnetic field \(d A\) vanishes somewhere, then the solutions concentrate around its zeroes. The known properties of \(\mathcal{E}\) are summarized in section~\ref{sectionLimiting}.

The result allows to obtain that the Lorentz force given by 
\eqref{eqLorentzDipole} vanishes in the semi-classical limit. Indeed, if \(x_* 
\in \Omega\) is a cluster point of the family 
\((x_\varepsilon)_{\varepsilon>0}\), then by taking a sequence we can assume 
that the sequence \((x_{\varepsilon_n})_{n \in \N}\) converges to \(x_*\) and, 
by the results in section~\ref{sectionAsymptotics}, that the translated 
rescaled sequence \((u_{\varepsilon_n} (x_{\varepsilon_n} + \varepsilon_n 
\cdot))_{n \in \N}\) converges to a solution \(v_*\) of 
the limiting problem \eqref{problemLimit}. By 
proposition~\ref{propositionLimitingLorentzDipole}, the 
equation~\eqref{eqLorentzDipole} is satisfied with \(q\) and \(\mu\) being the 
quantum mechanical charge and magnetic moment.

Mathematically, the existence of solutions to \eqref{problemMNLSE} is classical \cite{EstebanLions1989}. The study of the asymptotics of solutions as \(\varepsilon \to 0\) brings several problems. First the structure of the set of minimizers in \eqref{equationDefinitionE} is not known. In fact, there is no reason to believe that minimizers should be nondegenerate or unique up to translations, or even that they should be radial. We will thus develop arguments that do not depend on any structure of the set of groundstates of the limiting problem.

A second difficulty is that the strong magnetic field is large enough to be an 
obstacle to regularity estimates on rescaled solutions. To illustrate this, 
we observe that if, to fix the ideas, \(\Omega = \R^N\), then the function 
\(v_\varepsilon\) defined for \(y \in \R^N\) by \(v_\varepsilon (y) = 
u_\varepsilon (\varepsilon y)\) satisfies for each \(y \in \R^N\) the equation
\[
  -\Delta v_\varepsilon (y) - \frac{2 \imath}{\varepsilon^2} \scalprod{D v_\varepsilon (y)}{A (\varepsilon y)} - \frac{\imath d^* A (\varepsilon y)}{\varepsilon}  v_\varepsilon (y) + \frac{\abs{A (\varepsilon y)}^2}{\varepsilon^4} v_\varepsilon (y) + V (\varepsilon y) v_\varepsilon (y) = \abs{v_\varepsilon (y)}^{p - 2} v_\varepsilon (y).
\]
Even if we may assume by a suitable gauge transformation that \(\abs{A (\varepsilon y)} \le C \abs{\varepsilon y}\) for small \(y\), we still do not have locally uniformly bounded coefficients. 
In order to bypass this problem, we will arrange our proof in order to limit the use of regularity theory to estimates on the modulus \(\abs{v_\varepsilon}\) in \(L^\infty\) by the Kato inequality and by the De Giorgi--Nash--Moser regularity theory. In particular, instead of having compactness in the uniform norm, we will just have some sufficient condition for uniform convergence to \(0\).

A last challenge is that there is no notion of positive solutions for limiting problems like \eqref{problemLimit}. This rules out Liouville-type theorem based on comparison and prevents us in fact of using Liouville theorems in blowup arguments.

\bigbreak

Our second result is a local concentration result.

\begin{theorem}
\label{theoremLocal}
Assume that \(\frac{1}{2} - \frac{1}{N} < \frac{1}{p} < \frac{1}{2}\) and that either \(\Omega \subseteq \R^N\) is bounded, \(\frac{1}{p} < 1 - \frac{2}{N}\) or that  
\[
  \liminf_{\abs{x} \to \infty} V (x) \abs{x}^2 > 0.
\]
If \(\Lambda \subset \Omega\) is open, bounded and not empty, and satisfies
\[
  \inf_{\partial \Lambda} \mathcal{C} > \inf_{\Lambda} \mathcal{C},  
\]
if \(V \in C (\Bar{\Lambda})\), \(\inf_\Lambda V >0\), \(A \in C^1 (\Bar{\Lambda})\) and \(\Omega\) has a \(C^1\) boundary in a neighbourhood of \(\Bar{\Lambda}\),
then there exists a family of solutions \((u_\varepsilon)_{\varepsilon > 0}\) of \eqref{problemMNLSE} and a family of points \((x_\varepsilon)_{\varepsilon > 0}\) such that 
\begin{align*}
\lim_{\substack{R \to \infty\\ \varepsilon \to 0}} \norm{u_\varepsilon}_{L^\infty (\Omega \setminus B_\rho (x_\varepsilon))} &= 0,&
  & \lim_{\varepsilon \to 0} \varepsilon^{-N} \mathcal{F}_\varepsilon (u_\varepsilon)
  = \lim_{\varepsilon \to 0} \mathcal{C} (x_\varepsilon) = \inf_{\Lambda} \mathcal{C}.  
\end{align*}
\end{theorem}

In the two-dimensional case, the first assumption reduces to \(\Omega\) bounded or 
\[
 \liminf_{\abs{x} \to \infty} V (x) \abs{x}^2 > 0.
\]
It also follows from the assumption \(\inf_\Lambda V > 0\) and the diamagnetic inequality that 
\[
 \inf_{\Lambda} \mathcal{C} \ge \inf_{\Lambda} \mathcal{E} (V, 0) > 0.
\]
As for theorem~\ref{theoremGlobal}, the Lorentz force given by \eqref{eqLorentzDipole} vanishes in the semiclassical limit \(\varepsilon \to 0\).

When \(A = 0\), this reduces to the nonmagnetic case \citelist{\cite{MorozVanSchaftingen2009}\cite{MorozVanSchaftingen2010}}. In that case it has been shown that the assumption on the decay of \(V\) cannot be substantially improved.

Our construction of the solutions is variational. 
We follow the penalization scheme which was developped by M.\thinspace del Pino and P.\thinspace Felmer for the nonlinear Schr\"odinger equation
\citelist{\cite{delPinoFelmer1996}\cite{delPinoFelmer1997}\cite{delPinoFelmer1998}} and adapted to critical frequency for fast-decaying potentials \citelist{\cite{BonheureVanSchaftingen2006}\cite{BonheureVanSchaftingen2008}\cite{BonheureDiCosmoVanSchaftingen2012}\cite{DiCosmo2011}\cite{DiCosmoVanSchaftingen2013}\cite{MorozVanSchaftingen2010}\cite{YinZhang2009}} and 
to the nonlinear Schr\"odinger equation with a mild magnetic field \citelist{\cite{AlvesFigueiredoFurtado2011}\cite{CingolaniSecchi2005}}.

We combine both adaptations for the first time. 
After defining the penalization and proving existence of solutions to the penalized problem, we need to show that solutions are small enough in the penalized region to satisfy the original unpenalized problem. 
The classical strategy is to obtain asymptotic estimates on the action, which give some information about the decay of integrals on the solutions on balls of radius of the order of \(\varepsilon\).
In the mild magnetic field r\'egime, those can be improved by local uniform Schauder estimates \cite{CingolaniSecchi2005}; as mentioned above the coefficients of the rescaled linear operator are not controlled sufficiently to have uniform Schauder estimates.
When the concentration points are \emph{global minimizers} of the electric potential in the mild magnetic field r\'egime, the limiting functional bounds from below the penalized functional, allowing to show the strong convergence of rescaled solutions to a solution of the limiting problem in the energy space; uniform bounds can be derived by a suitable Moser iteration scheme on outer domains \cite{AlvesFigueiredoFurtado2011}. In the strong magnetic field r\'egime we rely instead on comparison principles for the modulus by Kato's inequality \cite{Kato1972}.

The lack of information concerning the limiting problem has forced us to prove the results with relying on the minimal properties that we could await from it. We do not even require the existence of solutions to the limiting problem; the only property that we use is the upper semicontinuity of the action of the limiting problem.

Our method is quite flexible, allowing us to treat unbounded domains, fast-decaying electric potentials and unbounded electromagnetic fields.

We have made two choices in the presentation that might be unusual in the community of analysts but that we think should highlight the geometrical features of the problem. First, we used derivatives and differential forms instead of gradients and vector fields. Next, we work on \(\C\) as a two-dimensional Euclidean vector field --- in particular \emph{all the scalar products are real} --- except for the multiplication by the imaginary unit \(\imath\) which can be thought in fact as the application of a skew-symmetric linear mapping.

\section{Construction of solutions to a penalized problem}

\subsection{Definition of the penalized problem}

In order to prepare the proof of theorem~\ref{theoremLocal}, we define and solve a penalized problem following the strategy of M.\thinspace{}del Pino and P.\thinspace{}Felmer \cite{delPinoFelmer1996} and its adaptation to critical potentials \citelist{\cite{BonheureVanSchaftingen2008}\cite{MorozVanSchaftingen2010}\cite{YinZhang2009}} and to the stationary magnetic nonlinear Schr\"odinger equation \citelist{\cite{CingolaniSecchi2005}\cite{AlvesFigueiredoFurtado2011}}. The reader only interested in the proof of theorem~\ref{theoremGlobal} can go directly to section~\ref{sectionAsymptotics}.

Without loss of generality, we assume that \(0 \in \Lambda\).
Following V.\thinspace{}Moroz and J.\thinspace{}Van Schaftingen \citelist{\cite{MorozVanSchaftingen2009}*{(12)}\cite{MorozVanSchaftingen2010}*{\S 3.1 and \S 6.1}\cite{BonheureDiCosmoVanSchaftingen2012}*{\S 3}\cite{DiCosmoVanSchaftingen2013}*{\S 2.1}}, the \emph{penalization potential} \(H : \Omega \to \R\) is defined at every point \(x \in \Omega\) by
\begin{equation*}
 H(x) \defeq  \frac{\chi_{\Omega \setminus \Lambda}(x) \bigl(\log\frac{\rho}{\rho_0} \bigr)^\beta}{4 \abs{x}^2 \bigl(\log \frac{\abs{x}}{\rho_0}\bigr)^{2+\beta}},
\end{equation*}
where the point \(x_0 \in \Lambda\) and the radius \(\rho > 0\) are choosen so that
\(\overline{B}_\rho (x_0) \subset \Lambda\), the parameters \(\beta > 0 \) and \(\rho_0 \in (0, \rho)\) are fixed, and \(\chi_{\Omega \setminus \Lambda}\) denotes the characteristic function of the set \(\Omega \setminus \Lambda\). 
If \(\Lambda = \Omega\), as it will be the case in the proof of 
theorem~\ref{theoremGlobal}, the penalization potential vanishes identically on 
its whole domain \(\Omega\).
 
By the classical Hardy inequality, the operator \(-\Delta - H\) satisfies a positivity principle \citelist{\cite{MorozVanSchaftingen2010}*{lemma~3.1}\cite{DiCosmoVanSchaftingen2013}*{lemma~2.1}}.

\begin{lemma}[Smallness and compactness of the penalization potential]
\label{lemmaPositivity}
There exists \(\Bar{\varepsilon} > 0\) such that for every \(\varepsilon \in (0, \Bar{\varepsilon}]\) and each function \(\varphi \in C^1_c (\Omega; \R)\), 
\begin{equation*}
 \int_{\Omega} \varepsilon^2 H \abs{\varphi}^2
 \le \int_{\Omega} \varepsilon^2 \abs{D \varphi}^2 + V \abs{\varphi}^2.
\end{equation*}
Moreover the corresponding embedding \(H^1_{V} (\Omega) \subset L^2 (\Omega, H (x)\dif x)\) is compact.
\end{lemma}

The weighted Sobolev space \(H^1_{V} (\Omega)\) is defined to be the completion of \(C^1_c (\Omega)\) with respect to the norm defined for every function \(\varphi \in C^1_c (\Omega)\) by
\[
  \Bigl(\int_{\Omega} \abs{D \varphi}^2 + V \abs{\varphi}^2 \Bigr)^\frac{1}{2}.
\]

\begin{proof}%
[Proof of lemma~\ref{lemmaPositivity}]
There exists a constant \(C>0\) such that for every test function \(\varphi \in C^1_c(\Omega)\)  \cite{MorozVanSchaftingen2010}*{lemma~6.1},
\begin{equation*}
\frac{1}{4}\int_{\Omega}\frac{\abs{\varphi (x)}^2}{\abs{x}^2\big(\log\frac{\abs{x}}{\rho_0}\big)^2}\dif x \le \int_{\Omega}\abs{D \varphi}^2+C\int_{B_\rho \setminus B_{\rho_0}} \abs{\varphi}^2;
\end{equation*}
the inequality follows since \(\inf_{B_\rho} V > 0\).
Since the embedding \(H^1_\mathrm{loc} (\Omega) \subset L^2_\mathrm{loc} (\Omega)\) is compact, the potential \(H\) is bounded and \(\lim_{\abs{x} \to \infty} H (x) \abs{x}^2 (\log \abs{x})^2 = 0\), we conclude that the embedding \(H^1_V (\Omega) \subset L^2 (\Omega, H (x)\dif x) \) is compact.
\end{proof}

For each \(\varepsilon > 0\), we define the \emph{penalized nonlinearity} \(g_\varepsilon : \Omega \times \C \to \C\) for every \((x, s) \in \Omega \times \C\) by \citelist{\cite{CingolaniSecchi2005}*{(21)--(22)}\cite{AlvesFigueiredoFurtado2011}*{(2.3)}}
\[
 g_\varepsilon (x, s)  
  \defeq \chi_{\Lambda} (x) \abs{s}^{p - 2} s 
    + \chi_{\Omega \setminus \Lambda} (x) \min \bigl(\varepsilon^2 \mu H (x), \abs{s}^{p - 2} \bigr) s.
\]
The penalized nonlinearity \(g_\varepsilon\) is variational, that is, for every \((x, s) \in \Omega \times \C\),
\[
 g_\varepsilon (x, s) = \nabla_s G_\varepsilon (x, s),
\]
where the function \(G_\varepsilon : \Omega \times \C \to \R\) is defined for every \((x, s) \in \Omega \times \C\) by 
\[
 G_\varepsilon (x, s) = \chi_{\Lambda} (x) \frac{\abs{s}^{p}}{p} + \chi_{\Omega \setminus \Lambda} (x)
 \int_0^{\abs{s}} \min \bigl(\mu \varepsilon^2 H (x) t, t^{p - 1} \bigr)\dif t.
\]

The penalized nonlinearity has the following properties: for every \(\varepsilon > 0\),
\begin{align}
\tag{$g_1$} & g_\varepsilon(x, s)=o(\abs{s})  \text{ as \(s\to 0^+\) uniformly on compact subsets of \(\R^N\)};\\
\tag{$g_2$} &\abs{g_\varepsilon(x, s)} \le \abs{s}^{p - 1} \text{ for every \((x, s) \in \Omega \times \C\)};\\
\tag{$g_3$} &\abs{g_\varepsilon(x, s)} \le \mu \varepsilon^2 H(x)\abs{s} \text{ for every \((x, s)\in\Lambda\times\C\)};\\
\tag{$g_4$} &2G_\varepsilon(x, s)\le \scalprod{s}{g_\varepsilon(x, s)} \text{ for every \((x, s)\in \Omega \times \C\)};\\
\tag{$g_5$} & p G_\varepsilon(x, s)\le \scalprod{s}{g_\varepsilon(x, s)} \text{ for every \((x, s)\in\Lambda\times\C\)};\\
\tag{$g_6$} & G_\varepsilon (x, s) > 0 \text{ for every \((x, s) \in \Omega \times \bigl(\C\setminus \{0\}\bigr)\)}.
\end{align}
We also denote by \(\mathsuper{g}_\varepsilon\) and \(\mathsuper{G}_\varepsilon\) the corresponding superposition operators, that is, for every function \(u : \Omega \to \R\), the functions \(\mathsuper{g}_\varepsilon (u) : \Omega \to \C\) and \(\mathsuper{G}_\varepsilon (u) : \Omega \to \R\) are defined for every \(x \in \Omega\) by
\begin{align*}
  \mathsuper{g}_\varepsilon (u) (x) & = g_\varepsilon (x, u (x)) &
  & \text{ and }&
  \mathsuper{G}_\varepsilon (u) (x) & = g_\varepsilon (x, u (x)).  
\end{align*}
Since \(g_\varepsilon\) and \(G_\varepsilon\) are Carath\'eodory functions, the measurability of \(u\) implies the measurability of the functions \(\mathsuper{g}_\varepsilon (u)\) and \(\mathsuper{G}_\varepsilon (u)\).

\subsection{Existence of solutions to the penalized problem}

The \emph{penalized functional} \(\mathcal{G}_\varepsilon\) is defined 
on the space \(H^1_{V, A/\varepsilon^2} (\Omega) (\Omega; \C)\) 
for every function \(u \in H^1_{V, A/\varepsilon^2} (\Omega)\) by
\[
 \mathcal{G}_\varepsilon (u) \defeq \frac{1}{2} \int_{\R^N} \varepsilon^2 \abs{D_{A/\varepsilon^2} u}^2 + V \abs{u}^2 -
 \int_{\R^N} \mathsuper{G}_\varepsilon (u).
\]
In contrast with the nonmagnetic case, the domain of the penalized functional \(\mathcal{G}_{\varepsilon}\) depends in general on the parameter \(\varepsilon\).

We shall construct weak solutions to the \emph{penalized problem}
\begin{equation}
  \label{problemPenalized}\tag{$\mathcal{Q}_\varepsilon$}
  \left\{
  \begin{aligned}
    -\varepsilon^2 \Delta_{A/\varepsilon^2} u_\varepsilon + V u_\varepsilon & = \mathsuper{g}_\varepsilon (u_\varepsilon)&
      & \text{in \(\Omega\)},\\
    u_\varepsilon & = 0 &
      & \text{on \(\partial \Omega\)},
  \end{aligned}
  \right.
\end{equation}
that is, \(u_\varepsilon \in H^1_{A/\varepsilon^2} (\Omega)\) and for every \(v \in H^1_{A/\varepsilon^2} (\Omega)\)
\[
 \int_{\Omega} \varepsilon^2 \scalprod{D_{A/\varepsilon^2} u_\varepsilon}{D_{A/\varepsilon^2} v} + V \scalprod{u_\varepsilon}{v} = \int_{\Omega} \scalprod{\mathsuper{g}_\varepsilon (u_\varepsilon)}{v}.
\]

We first prove that the penalized functional \(\mathcal{G}_\varepsilon\) is well-defined and continuously differentiable.

\begin{lemma}[Well-definiteness and continuous differentiability of the functional]
\label{lemmaFunctional}
The functional \(\mathcal{G}_{\varepsilon}\) is well defined and continuously differentiable on the space \(H^1_{V, A/\varepsilon^2} (\Omega)\). 
Its critical points are weak solutions of the penalized problem \eqref{problemPenalized}.
\end{lemma}

Before proving the lemma, we recall the diamagnetic inequality which is a powerful tool to study magnetic problems.

\begin{lemma}[Diamagnetic inequality (see for example \citelist{\cite{LiebLoss}*{theorem 7.21}\cite{EstebanLions1989}*{(2.3)}})]
\label{lemmaDiamagnetic}
If \(u \in H^1_{V, A/\varepsilon^2} (\Omega)\), then \(\abs{u} \in H^1_V (\Omega)\) and 
\[
  \abs{D \abs{u}} \le \abs{D_{A/\varepsilon^2} u}
\]
\end{lemma}

Thanks to the diamagnetic inequality, we prove a counterpart of the Hardy-type inequality of lemma~\ref{lemmaCoerciveness} in magnetic spaces.

\begin{lemma}[Smallness and compactness of the penalization potential on magnetic spaces]
\label{lemmaPositivityMagnetic}
For every \(\varepsilon \in (0, \Bar{\varepsilon}]\) and \(u \in H^1_{0, V, A/\varepsilon^2} (\Omega)\), we have \(u \vert_{\Omega \setminus \Lambda} \in L^2 (\Omega \setminus \Lambda; H (x)\dif x)\) and 
\begin{equation*}
 \int_{\Omega \setminus \Lambda} \varepsilon^2 H \abs{u}^2
 \le \int_{\R^N} \varepsilon^2 \abs{D_{A/\varepsilon^2} u}^2 + V \abs{u}^2.
\end{equation*}
Moreover the corresponding embedding \(H^1_{V, A/\varepsilon^2} (\Omega) \subset L^2 (\Omega \setminus \Lambda; H (x)\dif x)\) is compact.
\end{lemma}
\begin{proof}
The inequality follows from the corresponding statement for scalar functions lemma~\ref{lemmaPositivity} and the diamagnetic inequality (lemma~\ref{lemmaDiamagnetic}).

For the compactness of the embedding, we assume that \(u_n \weakto 0\) weakly 
in \(H^1_{V, A/\varepsilon^2} (\Omega)\) as \(n \to \infty\). 
By the previous inequality the sequence \((u_n)_{n \in \N}\) is bounded in \(L^2_\mathrm{loc} (\Omega)\), and thus, as the vector potential \(A\) is continuous, \((Du_n)_{n \in \N}\) is bounded in \(L^2_\mathrm{loc} (\Omega)\).
By the classical Rellich compactness theorem \(u_n \to 0\) strongly in \(L^2_\mathrm{loc} (\Omega)\) and thus \(\abs{u_n} \to 0\) strongly in \(L^2_\mathrm{loc} (\Omega)\) as \(n \to \infty\). 
By the diamagnetic inequality (lemma~\ref{lemmaDiamagnetic}), the sequence \((\abs{u_n})_{n \in \N}\) is bounded in \(H^1_V (\Omega)\), and thus \(\abs{u_n} \weakto 0\) weakly in \(H^1_V (\Omega)\) as \(n \to \infty\). By the compactness of the corresponding scalar embedding (lemma~\ref{lemmaPositivity}), \(\abs{u_n} \to 0\) strongly in \(L^2 (\Omega, H (x)\dif x)\)  as \(n \to \infty\) and we conclude that \(u_n \to 0\) strongly in \(L^2 (\Omega, H (x)\dif x)\) as \(n \to \infty\). 
\end{proof}

We also prove a compact Sobolev embedding with a control on the norm.

\begin{lemma}[Rescaled magnetic Sobolev inequality]
\label{lemmaRescaledSobolev}
There exists a constant \(C > 0\) such that for  every \(\varepsilon > 0\) and \(u \in H^1_{V, A/\varepsilon^2} (\Omega)\), then \(u \vert_{\Lambda} \in L^p (\Lambda)\) and 
\[
 \int_{\Lambda} \abs{u}^p \le \frac{C}{\varepsilon^{N (\frac{p}{2} - 1)}} \Bigl(\int_{\R^N} \varepsilon^2 \abs{D_{A/\varepsilon^2} u}^2 + V \abs{u}^2 \Bigr)^\frac{p}{2}.
\]
Moreover, the corresponding embedding \(H^1_{V, A/\varepsilon^2} (\Omega) \subset L^p (\Lambda)\) is compact.
\end{lemma}
\begin{proof}
Since \(\Lambda\) is bounded, \(\inf V > 0\) and \(V\) is uniformly continuous on \(\Lambda\), there exists a function \(\psi \in C^1 (\Omega)\) such that \(\psi \ge 0\) in \(\Omega\), \(\psi = 1\) on \(\Lambda\), \(\inf_{\supp \psi} V > 0\) and \(D \psi \in L^\infty (\Omega)\). We have then \(\abs{\psi u} \in H^1 (\R^N)\) and by the classical Sobolev embedding and the diamagnetic inequality of lemma~\ref{lemmaDiamagnetic}
\[
\begin{split}
\int_{\Lambda} \abs{u}^p \le \int_{\Omega} \abs{\psi u}^p&\le  \frac{C}{\varepsilon^{N(\frac{p}{2}- 1})}\Big(\int_{\R^N} \varepsilon^2\bigabs{D \abs{\psi u}}^2 + \abs{\psi u}^2 \Bigr)^\frac{p}{2}\\
& \le \frac{C'}{\varepsilon^{(\frac{p}{2}-1)N}}  \Bigl(\int_{\Omega} \varepsilon^2 \bigabs{D \abs{u}}^2 + V \abs{u}^2 \Bigr)^\frac{p}{2}\\
& \le \frac{C'}{\varepsilon^{(\frac{p}{2}-1) N}}  \Bigl(\int_{\Omega} \varepsilon^2 \abs{D_{A/\varepsilon^2} u}^2 + V \abs{u}^2 \Bigr)^\frac{p}{2}.
\end{split}
\]
The compactness is proved as in the proof of lemma~\ref{lemmaPositivityMagnetic}.
\end{proof}

\begin{proof}[Proof of lemma~\ref{lemmaFunctional}]
This follows from the properties $(g_2)$ and $(g_3)$, and the estimates in magnetic spaces  (lemmas~\ref{lemmaPositivityMagnetic} and \ref{lemmaRescaledSobolev}).
\end{proof}

Another propery of the functional that we need is its coerciveness: the norm of \(u\) in the space  \(H^1_{V, A/\varepsilon^2} (\Omega)\) is controlled by the functional \(\mathcal{G}_\varepsilon (u)\) and its radial derivative \(\dualprod{\mathcal{G}_\varepsilon' (u)}{u}\).

\begin{lemma}[Coerciveness of the functional]
\label{lemmaCoerciveness}
For every \(\varepsilon \in (0, \Bar{\varepsilon}]\) and each \(u \in H^1_{V, A/\varepsilon^2} (\Omega)\),
\[
  \Bigl(\frac{1}{2}- \frac{1}{p}\Bigr)(1 - \mu) \int_{\Omega} \varepsilon^2 \abs{D_{A/\varepsilon^2} u}^2 + V \abs{u}^2 \le \mathcal{G}_\varepsilon (u) - \frac{1}{p}\dualprod{\mathcal{G}_\varepsilon' (u)}{u}.
\]
\end{lemma}
\begin{proof}
We compute by definition of the functional \(\mathcal{G}_\varepsilon\),
\[
   \Bigl(\frac{1}{2}- \frac{1}{p}\Bigr) \Bigl(\int_{\R^N} \varepsilon^2 \abs{D_{A/\varepsilon^2} u}^2 + V \abs{u}^2\Bigr) -  \mathcal{G}_\varepsilon (u) + \frac{1}{p}\dualprod{\mathcal{G}_\varepsilon' (u)}{u}
   = \int_{\R^N} G_\varepsilon (u) - \frac{\scalprod{g_\varepsilon (u)}{u}}{p}. 
\]
In view of first the properties $(g_5)$ and $(g_4)$, and then the property $(g_3)$, we have
\[
  \int_{\R^N} G_\varepsilon (u) - \frac{\scalprod{g_\varepsilon (u)}{u}}{p} \le \Bigl(\frac{1}{2} - \frac{1}{p}\Bigr) \int_{\R^N \setminus \Lambda} \scalprod{g_\varepsilon (u)}{u}\le \Bigl(\frac{1}{2} - \frac{1}{p}\Bigr) \mu \int_{\R^N \setminus \Lambda} \varepsilon^2 H \abs{u}^2.
\]
and the conclusion comes from the smallness property of the penalization potential (lemma~\ref{lemmaPositivityMagnetic}).
\end{proof}

We now show the existence of a suitable critical point of the functional \(\mathcal{G}_\varepsilon\).

\begin{proposition}[Existence of solutions to the penalized problem]
\label{propositionExistencePenalized}
For every \(\varepsilon>0\), there exists a solution \(u_\varepsilon \in H^1_{0, V, A/\varepsilon^2} (\Omega) (\R^N; \C)\) of the penalized problem \eqref{problemPenalized} such that 
\[
 \mathcal{G}_\varepsilon (u) = c_\varepsilon\defeq \inf_{\gamma\in\Gamma_\varepsilon}\max_{t\in[0, 1]} \mathcal{G}_\varepsilon\big(\gamma(t)\big),
\]
where 
\[
\Gamma_\varepsilon\defeq \big\{\gamma\in C\big([0, 1]; H^1_{0, V, A/\varepsilon^2} (\Omega) \big)\mid \gamma(0)=0\text{ and } \mathcal{G}_\varepsilon \big(\gamma(1)\big)<0\big\}. 
\]
\end{proposition}

It is known in critical point theory, that it suffices to prove that the minimax level is nondegenerate \(c_\varepsilon \in (0, \infty)\) and that the functional \(\mathcal{G}_\varepsilon\) satisfies the Palais--Smale condition \citelist{\cite{AmbrosettiRabinowitz1973}*{theorem 2.1}\cite{Rabinowitz1986}*{theorem 2.2}\cite{Struwe2008}*{theorem 6.1}\cite{Willem1996}*{theorem 2.10}}.

\begin{lemma}[Nondegeneracy of the critical level]
For every \(\varepsilon \in (0, \Bar{\varepsilon}]\), 
\[
 0 < c_\varepsilon < \infty.
\]
\end{lemma}
\begin{proof}
In order to prove that \(c_\varepsilon < \infty\), since the set \(\Lambda\) is open and not empty, we take \(u \in H^1_{0, V, A/\varepsilon^2} (\Omega) \setminus \{0\}\) such that \(u = 0\) on \(\Omega \setminus \Lambda\). By the property \((g_5)\), we have for every \((x, s) \in \Lambda \times \C\) and \(t \ge 1\), 
\(
  G (x, ts) \ge t^p G (x, s),
\)
and thus by $(g_6)$, for every \((x, s) \in \Lambda \times \C \setminus \{0\}\),  \(\limsup_{t \to \infty} t^{-p} G (x, t s) > 0\). Since \(p > 2\), this implies that \(\lim_{t \to \infty} \mathcal{G}_\varepsilon (t u) = - \infty\). In particular \(\Gamma_\varepsilon \ne \emptyset\) and \(c_\varepsilon < \infty\).

For the other inequality, we observe that for every \(u \in H^1_{V, A/\varepsilon^2} (\Omega)\), by the properties $(g_2)$ and $(g_3)$,
\[
  \mathcal{G}_\varepsilon (u) 
  \ge 
  \frac{1}{2} \int_{\Omega} \varepsilon^2 \abs{D_{A/\varepsilon^2} u}^2 + V \abs{u}^2  
  - \frac{1}{p} \int_{\Lambda} \abs{u}^{p} - \frac{\mu}{2} \int_{\Omega \setminus \Lambda} H \abs{u}^2.
\]
In view of the Hardy and Sobolev inequalities in magnetic Sobolev spaces (lemmas~\ref{lemmaPositivityMagnetic} and \ref{lemmaRescaledSobolev}), we have 
\[
  \mathcal{G}_\varepsilon (u) 
  \ge \frac{1 - \mu}{2} \int_{\Omega} \varepsilon^2 \abs{D_{A/\varepsilon^2} u}^2 + V \abs{u}^2 
  - \frac{C}{\varepsilon^{(p - 1)N}}\Bigl(\int_{\Omega} \varepsilon^2 \abs{D_{A/\varepsilon^2} u}^2 + V \abs{u}^2 \Bigr)^\frac{p}{2}.
\]
It follows since \(p > 2\) that \(0\) is a strict local minimizer of \(\mathcal{G}_\varepsilon\) (with respect to the strong topology of \(H^1_{V, A/\varepsilon^2} (\Omega)\)) and thus  we conclude that \(c_\varepsilon > 0\).
\end{proof}

Next, we prove that the functional \(\mathcal{G}_\varepsilon\) satisfies the Palais--Smale condition.

\begin{lemma}[Palais--Smale condition]
\label{lemmaPalaisSmale}
Let \((u_n)_{n \in \N}\) be a sequence in \(H^1_{0, V, A/\varepsilon^2} (\Omega)\) such that 
\begin{align*}
\mathcal{G}_\varepsilon(u_n)&\to c,&
&\text{ and }&
\mathcal{G}_\varepsilon'(u_n)\to 0 \text{ in \(\bigl(H^1_{0, V, A/\varepsilon^2} (\Omega)\bigr)^*\)}, 
\end{align*}
then, up to a subsequence, the sequence \((u_n)_{n \in \N}\) converges strongly in \(H^1_{V, A/\varepsilon^2} (\Omega)\).
\end{lemma}

Compared to other penalizations choices \citelist{\cite{delPinoFelmer1996}\cite{AlvesFigueiredoFurtado2011}\cite{CingolaniSecchi2005}}, the use of the penalization potential simplifies considerably the proof of the Palais--Smale condition \cite{MorozVanSchaftingen2010}*{remark~3.5}. Indeed \(\mathcal{G}_\varepsilon'\) is a compact perturbation of the duality map from \(H^1_{0, V, A/\varepsilon^2} (\Omega)\) to \(\bigl(H^1_{0, V, A/\varepsilon^2} (\Omega)\bigr)^*\) and the bounded Palais--Smale condition follows by a general argument \citelist{\cite{Struwe2008}*{proposition 2.2}\cite{Rabinowitz1986}*{proof of proposition B.35}}.

\begin{proof}[Proof of lemma~\ref{lemmaPalaisSmale}]
By the coerciveness of the functional (lemma~\ref{lemmaCoerciveness}), the sequence \((u_n)_{n \in \N}\) is bounded in \(H^1_{V, A/\varepsilon^2} (\Omega)\). 
Hence, there exists \(u \in H^1_{V, A/\varepsilon^2} (\Omega)\) such that, up to a subsequence \(u_n \weakto u\) weakly in \(H^1_{0, V, A/\varepsilon^2} (\Omega)\) as \(n \to \infty\). 
By the compactness part of the embeddings of magnetic spaces (lemma~\ref{lemmaRescaledSobolev} and  lemmas~\ref{lemmaPositivityMagnetic}), since \(\frac{1}{2} - \frac{1}{N} < \frac{1}{p} < \frac{1}{2}\), \(u_n \to u\) strongly as \(n \to \infty\) in \(L^2 (\Omega \setminus \Lambda; H (x)\dif x)\) and in \(L^{p}(\Lambda)\).
Therefore, by the properties $(g_2)$ and $(g_3)$, \(g_\varepsilon (u_n) \to g_\varepsilon (u)\) in \(L^{\frac{p}{p - 1}}(\Lambda)\) and in \(L^2 (\Omega \setminus \Lambda; \dif x/H (x))\)
as \(n \to \infty\). 
In particular, for every \(\varphi \in H^1_{V, A/\varepsilon^2} (\Omega;\C)\), we have
\[
  \begin{split}
    \int_{\Omega} \varepsilon^2 \scalprod{D_{A/\varepsilon^2} u}{D_{A/\varepsilon^2} \varphi} + V \scalprod{u}{\varphi}
      &= \lim_{n \to \infty} \int_{\Omega} \varepsilon^2 \scalprod{D_{A/\varepsilon^2} u_n}{D_{A/\varepsilon^2} \varphi} + V \scalprod{u_n}{\varphi} \\
      & = \lim_{n \to \infty} \int_{\Omega} \scalprod{\mathsuper{g}_\varepsilon (u_n)}{\varphi} = \int_{\Omega} \scalprod{\mathsuper{g}_\varepsilon (u)}{\varphi},
  \end{split}
\]
and thus \(u\) solves \eqref{problemPenalized}, or by lemma~\ref{lemmaFunctional}, \(\mathcal{G}_\varepsilon' (u) = 0\).
We observe now that since \(\mathcal{G}_\varepsilon' (u_n) \to 0\) as \(n \to \infty\), since \(\mathcal{G}_\varepsilon' (u) = 0\) and since the sequence \((u_n)_{n \in \N}\) is bounded, as \(n \to \infty\),
\begin{multline*}
 \int_{\Omega} \varepsilon^2 \abs{D_{A/\varepsilon^2} (u_n - u)}^2 + V \abs{u_n - u}^2\\
 = \dualprod{\mathcal{G}_\varepsilon' (u_n)-\mathcal{G}_\varepsilon' (u)}{u_n - u}
 +  \int_{\R^N} \scalprod{u_n -u}{g_\varepsilon(u_n)-g_\varepsilon(u)} \to 0.\qedhere
\end{multline*}
\end{proof}

\section{Properties of the limiting functional}
\label{sectionLimiting}

In this section we study the properties of the limiting problem \eqref{problemLimit}.
The existence of solutions to this problem has been studied in the seminal work of M.\thinspace{}Esteban and P.-L.\thinspace{}Lions \cite{EstebanLions1989}. The next result is a reformulation of a part of their results \cite{EstebanLions1989}*{theorem 3.1}.

\begin{proposition}%
[Existence of groundstates for the limiting problem]%
\label{propositionExistenceLimiting}
Let \(A_* \in \Lin (\R^N; \bigwedge^1 \R^N)\) and \(V_* > 0\), there exists \(v_* \in H^1 (A)\)  that satisfies the limiting equation \eqref{problemLimit} and such that 
\[
  \mathcal{I}_{V_*, A_*} (v_*) = \mathcal{E} (V_*, A_*).
\]
Moreover, 
\[
 \mathcal{E} (V_*, A_*) = \bigl(\tfrac{1}{2} - \tfrac{1}{p}\bigr)
 \inf_{v \in H^1_{V_*,A_*} (\R^N)} \mathcal{S}_{V_*, A_*} (v)^\frac{p}{p - 2}.
\]
where the Sobolev functional \(\mathcal{S}_{V_*, A_*} : H^1_A (\R^N) \setminus \{0\} \to \R\) is defined for every \(v \in H^1_A (\R^N)\) by
\[
 \mathcal{S}_{V_*, A_*} (v) \defeq \frac{\displaystyle \int_{\R^N} \abs{D_{A_*} v}^2 + V_* \abs{v}^2}{\displaystyle \Bigl(\int_{\R^N} \abs{v}^p \Bigr)^\frac{2}{p}}.
\]
\end{proposition}
\begin{proof}
The existence of a minimizer of a minimizer of the functional \(\mathcal{S}_{V_*, A_*}\) was proved by M.\thinspace{}Esteban and P.-L. Lions \cite{EstebanLions1989}*{theorem 3.1}. 
Up to multiplication by a positive constant, it can be chosen to satisfy the equation \eqref{problemLimit}.

Finally, if \(v \in H^1 (A)\), and \(\dualprod{\mathcal{I}_{V_*, A_*} (v)}{v} = 0\), then 
\[
 \mathcal{I}_{V_*, A_*} (u)
 = \bigl(\tfrac{1}{2} - \tfrac{1}{p} \bigr)\bigl(\mathcal{S}_{V_*, A_*} (v)\bigr)^{\frac{p}{p - 2}},
\]
and thus we conclude that 
\[
 \mathcal{E} (V_*, A_*) = \bigl(\tfrac{1}{2} - \tfrac{1}{p}\bigr) \inf_{v \in H^1_A (\R^N)} \mathcal{S} (v)^\frac{p}{p - 2}.\qedhere
\]
\end{proof}

The results of M.\thinspace{}Esteban and P.-L. Lions covers in fact the more general case \(V_* > -\abs{dA_*}\).

\begin{lemma}[Characterization of the limiting action by smooth test functions]
\label{lemmaCharacterizationE}
For every \(V_* \in \R^+\) and \(A_* \in \Lin (\R^N; \bigwedge^1 \R^N)\), 
\[
  \mathcal{E} (A) = \inf_{v \in C^1_c (\R^N; \C) \setminus \{0\}} \sup_{t > 0} \mathcal{I}_{V_*, A_*} (tv)
  = \bigl(\tfrac{1}{2} - \tfrac{1}{p}\bigr)
 \inf_{v \in C^1_c (\R^N; \C)\setminus \{0\}} \mathcal{S}_{V_*, A_*} (v)^\frac{p}{p - 2}.
\]
\end{lemma}

This lemma is proved by a direct density argument. The main interest is that the class of functions appearing in the variational principle \emph{does not depend} on the magnetic potential \(A_*\).

\begin{proposition}[Properties of the action of the limiting problem]
\label{propositionPropertiesLimiting}
Let \(V_* \in \R^+\) and \(A_* \in \Lin (\R^N; \bigwedge^1 \R^N)\).
\begin{enumerate}[(i)]
  \item (invariance under isometries) \label{itInvariance} If \(\iota : \R^N \to \R^N\) is a linear isometry, then 
  \[
    \mathcal{E} (V_*, \iota_\# A_*) = \mathcal{E} (V_*, A_*). 
  \]
  \item (gauge invariance) \label{itGauge} If \(dA_*=d\Tilde{A}_*\), then
  \[
    \mathcal{E} (V_*, \Tilde{A}_*) = \mathcal{E} (V_*, A_*) ,
  \]
   \item (scaling of the electromagnetic potential)\label{itScaling} For every \(\lambda > 0\),
  \[
    \mathcal{E} (\lambda^2 V_*, \lambda A_*) = \lambda^{\frac{4}{p - 2} - (N - 2)} \mathcal{E} (V_*, A_*),
  \]
  \item (monotonicity with respect to the electric potential) \label{itMonotonicity} If \(\Tilde{V}_* > V_*\), then 
  \[
    \mathcal{E} (\Tilde{V}_*, A_*) > \mathcal{E} (V_*, A_*),
  \]
  \item (diamagnetic inequality) \label{itDiamagnetic} If \(dA_* \ne 0\), then 
  \[
    \mathcal{E} (V_*, A_*) > \mathcal{E} (V_*, 0),
  \]
\end{enumerate}
\end{proposition}

In this statement \(\iota_\# A\) denotes the \emph{pull-back} of the differential form \(A\) by the map \(\iota\): for every \(x \in \R^N\) and \(v \in \R^N\),
\[
  \iota_\# A_* (x) [v] = A (\iota (x))\bigl[D \iota[v]\bigr].
\]
In our particular case, since \(\iota\) is linear it can be written simply as \(\iota_\# A (x) [v] A_* (\iota (x))\bigl[\iota [v]\bigr]\). We also note that since \(A_*\) is linear, \(d A_* (x) [v, w] = A (v)[w] - A (w)[v]\).

\begin{proof}[Proof of proposition~\ref{propositionPropertiesLimiting}]
The invariance under isometries \eqref{itInvariance} follows from the fact that \(D_{\iota_\# A_*} (v \circ \iota) = (D_{A_*} v) \circ \iota\).
For the gauge property \eqref{itGauge}, recall that for every \(y \in \R^N\) and every \(k, h \in \R^N\), 
\[
  d A_* (y) [h, k] = A_* (h)[k] - A_* (k)[h].
\]
and therefore
\[
  \Tilde{A}_* (y)[k] = A_* (y) - D \varphi (y), 
\]
where the quadratic form \(\varphi : \R^N \to \R\) is defined for every \(y \in \R^N\) by
\[
  \varphi (y) \defeq A_* (y)[y] - \Tilde{A}_* (y)[y].
\]
The conclusion follows then from the fact that for every \(v \in H^1_{V_*, A_*} (\R^N)\), \(\e^{\imath \varphi} v \in H^1_{V_*, \Tilde{A}_*} (\R^N)\) and
\[
  D_{\Tilde{A}_*} (\e^{\imath \varphi} v) = D_{A_*} v
\]
and the statement follows.

The scaling statement \eqref{itScaling} follows by noting that for every \(v \in H^1_{\lambda^2 V_*, \lambda A_*} (\R^N)\), the function \(v_\lambda : \R^N \to \C\) defined for every \(y \in \R^N\) by 
\[                                                                                                                                                                                
  v_\lambda (y) \defeq \lambda^\frac{2}{p - 2} v (\lambda y)  
\]
is in the space \(H^1_{V_*, A_*} (\R^N)\) and 
\[
  \mathcal{I}_{\lambda^2 V_*, \lambda A_*} (v_\lambda) 
  = \lambda^{\frac{4}{p - 2} - (N - 2)} \mathcal{I}_{V_*, A_*} (v).
\]

For the monotonicity with respect to the electric field \eqref{itMonotonicity}, for every \(v \in H^1_{\Tilde{V}_*, A_*} (\R^N)\), we have \(v \in H^1_{V_*, A_*} (\R^N)\) and 
\[
  \mathcal{I}_{\Tilde{V}_*, A_*} (v) > \mathcal{I}_{V_*, A_*} (v) .
\]
Since the infimum \(\mathcal{E} (\Tilde{V}_*, A_*)\) is achieved by proposition~\ref{propositionExistenceLimiting}, we conclude that \(\mathcal{E} (\Tilde{V}_*, A_*) > \mathcal{E} (V_*, A_*)\).

Finally, for the diamagnetic inequality \eqref{itDiamagnetic}, for every  \(v \in H^1_{V_*, A_*} (\R^N)\), by the diamagnetic inequality (lemma~\ref{lemmaDiamagnetic}), 
\[
  \mathcal{I}_{V_*, 0} (\abs{v}) \le \mathcal{I}_{V_*, A_*} (v).
\]
It follows then that \(\mathcal{E} (V_*, 0) \le \mathcal{E} (V_*, A_*)\).
If there is equality, then by proposition~\ref{propositionExistenceLimiting} \(\mathcal{E} (V_*, A_*)\) is achieved by some \(v_* \in H^1_{V_*, A_*} (\R^N)\). Moreover, 
\begin{equation}
\label{diamagneticequality}
  \int_{\R^N} \abs{D \abs{v_*}}^2 + V_* \abs{v_*}^2
  =\int_{\R^N} \abs{D_{A_*} v_*}^2 + V_* \abs{v_*}^2.
\end{equation}
In particular \(\abs{v}_*\) achieves \(\mathcal{E} (V_*, 0)\). By the corresponding Euler-Lagrange equation \eqref{problemLimit}, regularity theory and the strong maximum principle the function \(\abs{v_*}\) is locally bounded away from \(0\). Therefore, there exists \(\varphi \in H^1_{\mathrm{loc}} (\R^N)\) such that \(v_* = \e^{\imath \varphi} \abs{v_*}\). 
By \eqref{diamagneticequality}, we have \(D \varphi + A_* = 0\) almost everywhere in \(\R^N\), which is a contradiction if \(d A_* \ne 0\).
\end{proof}

Finally, the action of the limiting problem is a continuous function of the electromagnetic potential. 

\begin{proposition}[Continuity of the action of the limiting problem]
\label{propositionEcontinuous}
The function \(\mathcal{E} : \R^+ \times \Lin (\R^N; \bigwedge^1 \R^N) \to \R^+\) is continuous.
\end{proposition}

Without a magnetic field, proposition~\ref{propositionEcontinuous} is due to P.\thinspace{}Rabinowitz \cite{Rabinowitz1992}.
We mention to the reader who is interested in the proof of theorem~\ref{theoremGlobal} or theorem~\ref{theoremLocal}, that the latter proofs only rely on the upper semicontinuity which is the most easy part of the proof of proposition~\ref{propositionEcontinuous}.

\begin{proof}%
[Proof of proposition~\ref{propositionEcontinuous}]
By lemma~\ref{lemmaCharacterizationE}, the function \(\mathcal{E}\) is upper semicontinuous as by lemma~\ref{lemmaCharacterizationE} an infimum of continuous functions on \(\R^+ \times \Lin (\R^N; \bigwedge^1 \R^N)\).

Assume that the sequence \(((V_n, A_n))_{n \in \N}\) converges to \((V_*, A_*)\) in \(\R^+ \times \Lin (\R^N; \bigwedge^1 \R^N)\).
By proposition~\ref{propositionExistenceLimiting}, the problem \((\mathcal{R}_{V_n, A_n})\) 
has a weak solution \(v_n \in H^1_{V_n, A_n}\) such that \(\mathcal{I}_{V_n, A_n} (v_n) = \mathcal{E} (V_n, A_n)\).
Since \(v_n \in H^1_{V_n, A_n} (\R^N) \setminus \{0\}\) solves the problem \((\mathcal{R}_{V_n, A_n})\) we have
\begin{equation}
\label{eqLpH1}
  \int_{\R^N} \abs{D_{A_n} v_n}^2 + V_n \abs{v_n}^2 = \int_{\R^N} \abs{v_n}^p = \tfrac{2 p}{p - 2} \mathcal{I}_{V_n, A_n} (v_n),
\end{equation}
and thus by the upper semicontinuity proved above
\[
  \limsup_{n \to \infty} \int_{\R^N} \abs{D_{A_n} v_n}^2 + V_n \abs{v_n}^2
  \le \tfrac{2 p}{p - 2} \mathcal{I}_{V_*, A_*} (v_*).
\]

By an inequality of P.-L. Lions \cite{Lions1984CC2}*{lemma I.1} (see also \citelist{\cite{Willem1996}*{lemma 1.21}\cite{MorozVanSchaftingenGround}*{lemma 2.3}\cite{VanSchaftingen}} and by the diamagnetic inequality (lemma~\ref{lemmaDiamagnetic}), we have 
\[
\begin{split}
  \int_{\R^N} \abs{v_n}^{p}
  & \le C \Bigl(\sup_{a \in \R^N} \int_{B_{R_n} (a)} \abs{v_n}^{p}\Bigr)^{1 - \frac{2}{p}}  \int_{\R^N} \abs{D \abs{v_n}}^2 + \abs{v_n}^2\\
  & \le C \Bigl(\sup_{a \in \R^N} \int_{B_{R_n} (a)} \abs{v_n}^{p}\Bigr)^{1 - \frac{2}{p}}  \int_{\R^N} \abs{D_{A_n} v_n}^2 + V_n \abs{v_n}^2,
\end{split}
\]
where \(R_n \defeq 1/\sqrt{V_n}\). Since \(v_n \ne 0\), we deduce that 
\[
  C \Bigl(\sup_{a \in \R^N} \int_{B_{R_n} (a)} \abs{v_n}^{p}\Bigr)^{1 - \frac{2}{p}} \ge 1
\]
We deduce that there exists a sequence of points \((a_n)_{n \in \N}\) in \(\R^N\) such that 
\[
  \liminf_{n \to \infty} \int_{B_{R_n} (a_n)} \abs{u_n}^p > 0.
\]
Since \(V_* \ne 0\), \(\limsup_{n \to \infty} R_n \le R_* \defeq 1/\sqrt{V_*}\) and thus 
\[
  \liminf_{n \to \infty} \int_{B_{2 R_*} (a_n)} \abs{u_n}^p > 0
\]
We define now the function \(\Tilde{v}_n : \R^N \to \C\) for every \(y \in \R^N\) by 
\[
  \Tilde{v}_n (y) \defeq \e^{-\imath A_n (a_n)[y]} v_n (y - a_n).
\]
We have \(\Tilde{v}_n \in H^1_{V_n, A_n} (\R^N)\) and for every \(y \in \R^N\),
\[
  D_{A_n} \Tilde{v}_n (y) =\e^{-\imath A_n (a_n)[y]} (D_{A_n} v_n) (y - a_n)
\]
and 
\begin{equation}
\label{eqLiminfTildev}
  \liminf_{n \to \infty} \int_{B_{2 R^*}} \abs{\Tilde{v}_n}^p > 0.
\end{equation}
Up to a subsequence, we can assume that the sequence \((\Tilde{v}_n)_{n \in \N}\) converges weakly to some function \(v_*\) in \(L^2 (\R^N; \C)\) and the sequence \((D_{A_n} \Tilde{v}_n)_{n \in \N}\) converges weakly to \(g_*\) in \(L^2 (\R^N; \Lin (\R^N; \C))\).
Since the sequence \((A_n)_{n \in \N}\) converges to \(A_*\), the sequence \((D \Tilde{v}_n)_{n \in \N}\) converges weakly to \(g_* - \imath A_* v_*\) and thus \(g_* = D_{A_*} v_*\).
By Rellich's theorem, the sequence \((\Tilde{v}_n)_{n \in \N}\) converges strongly to \(v_*\) in \(L^p_\mathrm{loc} (\R^N)\), and thus by the strict inequality \eqref{eqLiminfTildev} the function \(v_*\) is a nontrivial solution to \eqref{problemLimit} and
\[
  \mathcal{I}_{V_*, A_*} (v_*) = \bigl(\tfrac{1}{2} - \tfrac{1}{p}\bigr) \int_{\R^N} \abs{D_{A_*} v_*}^2 + V_* \abs{v_*}^2. 
\]
Finally, we conclude by lower semicontinuity of the norm under the convergence in \(L^2\) that 
\begin{equation*}
\begin{split}
  \liminf_{n \to \infty} \mathcal{E} (V_n, A_n)
  &= \liminf_{n \to \infty} \mathcal{I} (V_n, A_n)
  = \liminf_{n \to \infty} \bigl(\tfrac{1}{2} - \tfrac{1}{p}\bigr)\int_{\R^N} \abs{D_{A_n} v_n}^2 + V_n \abs{v_n}^2\\
  &= \liminf_{n \to \infty} \bigl(\tfrac{1}{2} - \tfrac{1}{p}\bigr)\int_{\R^N} \abs{D_{A_n} \Tilde{v}_n}^2 + V_n \abs{\Tilde{v}_n}^2\\
  &\ge \bigl(\tfrac{1}{2} - \tfrac{1}{p}\bigr) \int_{\R^N} \abs{D_{A_*} v_*}^2 + V_* \abs{v_*}^2
  \ge \mathcal{E} (V_*, A_*).
\end{split}
\end{equation*}
Since the function \(\mathcal{E}\) is upper and lower semicontinuous, it is continuous.
\end{proof}

Finally minimal points of the concentration functions can be characterized as points at which the Lorentz force \eqref{eqLorentzDipole} with the quantum mechanical charge and magnetic moment vanishes.

\begin{proposition}
\label{propositionLimitingLorentzDipole}
Assume that \(V \in C^1 (\Omega; \R)\) and \(A \in C^2 (\Omega;\bigwedge^1 \R^N)\).
If \(x_* \in \Omega\) is a local minimum point of \(\mathcal{E} \defeq \mathcal{C} (V, A)\) and 
\(v_* \in H^1_{D A (x_*)} (\R^N)\) achieves \(\mathcal{C} (V (x_*), A (x_*))\),
then
\[
  q\, dV (x_*) + \dualprod{d A}{\mu} (x_*)=0,
\]
where the charge \(q \in \R\) is defined by
\[ 
  q \defeq \int_{\R^N} \abs{v_*}^2
\]
and the magnetic moment \(\mu \in \bigwedge^2 \R^N\) is the bivector defined by
\[
  \mu \defeq \frac{1}{2}  \int_{\R^N} \scalprod{y \wedge \nabla_{D A_* (x)} v_* (y)}{\imath v_*(y)}.
\] 
\end{proposition}
\begin{proof}
% By gauge invariance (proposition~\ref{propositionPropertiesLimiting}, we have when \(x\) is close enough of \(x_*\),
% \[
%   \mathcal{C} (V (x_*), d A (x_*)/2)  = \mathcal{C} (V (x_*), D A (x_*)) =
%   \le \mathcal{C} (V (x_*), D A (x_*)) = \mathcal{C} (V (x), d A (x)/2).
% \]
% 
% 
% 
% First, we observe that by regularity theory for groundstates of \eqref{problemLimit}, the groundstate \(v_* \in H^1_{V (x_*), D A (x_*)} (\R^N)\),
% \[
%   \int_{\R^N} \abs{v_* (y)}^2 \abs{y}^2 \dif y < \infty
% \]
% and thus for every \(x \in \Omega\), \(v_* \in H^1_{V (x), D A (x)} (\R^N)\).
% 
% In view of the characterization of the groundstate level given in proposition~\ref{propositionExistenceLimiting}, we have, in view of the gauge invariance (proposition~\ref{pr}) for every \(x \Omega\),
% \[
%   \mathcal{C} (V (x_*), A (x_*)) \le \mathcal{S}_{V (x_*), d A (x_*)/2} (v_*)
% \]
% and thus
% \[
%   \frac{\dif}{\dif x}\mathcal{S}_{V_* (x), d A_* (x)/2} (v_*) \Big\vert_{x = x_*} = 0.
% \]
% Therefore, by a direct computation
% \[
%   0 = \frac{\dif}{\dif x} \int_{\R^N} \abs{D_{dA (x)/2} v_*}^2 + V (x) \abs{v_*}^2 \Big\vert_{x = x_*}
%   =\int_{\R^N} \scalprod{D_{d A (x_*)/2}v_*}{\imath D (d A) (x_*) v_*} = d\dualprod{d A}{\mu} (x_*).\qedhere
% \]
% 

First, we observe that by regularity theory for groundstates of \eqref{problemLimit}, the groundstate \(v_* \in H^1_{V (x_*), D A (x_*)} (\R^N)\) satisfies
\[
  \int_{\R^N} \abs{v_* (y)}^2 \abs{y}^2 \dif y < \infty
\]
We define for each \(x \in \Omega\) the function \(v_x : \R^N \to \C\) for every \(y \in \R^N\) by
\[
  v_x (y) = e^{\imath  (DA (x_*) [y, y] - DA (x)[y, y])/2} v_* (y).
\]
% One has 
% \[
% \begin{split}
%   D_{D A (x)} v_x [w]&= e^{\imath  ((A (x_*)[y, y] - A (x) [y, y])/2} (D v_* (y)[w] + \imath(DA (x_*)[w, y] + DA (x_*)[y, w] - DA (x)[w, y] + DA (x)[y, w])v_* (y)/2)\\
%   &=e^{\imath  ((A (x)[y, y] - A (x_* [y, y]))/2} (D_{D A (x_*)} v_* (y) [w] + \imath (d A (x)[y, w]- d A (x_*) [y, w]).
% \end{split}
% \]
% and thus 
We compute that 
\[
  D_{D A (x)} v_x (y) = e^{\imath  (DA (x_*)[y, y] - DA (x) [y, y])/2} \bigl(D_{D A (x_*)} v_* + \imath (d A (x)[y] - d A (x_*) [y])v_* (y)/2\bigr)
\]
and thus for every \(x \in \Omega\), \(v_* \in H^1_{V (x), D A (x)} (\R^N)\).

% 
% 
% 
% In view of the characterization of the groundstate level given in proposition~\ref{propositionExistenceLimiting}, we have, in view of the gauge invariance (proposition~\ref{pr}) for every \(x \Omega\),
% \[
%   \mathcal{C} (V (x_*), A (x_*)) \le \mathcal{S}_{V (x_*), d A (x_*)/2} (v_*)
% \]
Since \(v_*\) achieves the minimum of the Sobolev quotient \(\mathcal{S}_{V_* (x), d A_* (x)/2}\) and by definition of the limiting action \(\mathcal{E}\),
\[
  \frac{\dif}{\dif x}\mathcal{S}_{V_* (x), d A_* (x)/2} (v_*) \Big\vert_{x = x_*} = 0.
\]
Therefore, by a direct computation
\[
\begin{split}
  0 & = \frac{\dif}{\dif x} \int_{\R^N} \abs{D_{D A (x_*)} v_*}^2 + V (x) \abs{v_*}^2 \Big\vert_{x = x_*}
  =\int_{\R^N} \scalprod{D_{D A (x_*)}v_*}{\imath D (d A) (x_*) v_*} + \int_{\R^N} D V (x_*) \abs{v_*}^2 \\
  &=\int_{\R^N} D(d A) (x_*)[\imath y v_* (y), \nabla_{D A (x_*)} v_* (y)]\dif y+ \int_{\R^N} D V (x_*) \abs{v_*}^2\\
  &=\frac{1}{2} \int_{\R^N}  \bigl(D (d A) (x_*)[\imath y v_* (y), \nabla_{D A (x_*)} v_* (y)]-D (d A) (x_*)[\nabla_{D A (x_*)} v_* (y), \imath y v_* (y)]\bigr)\dif y\\
  & \hspace{32em}+ \int_{\R^N} D V (x_*) \abs{v_*}^2\\
  &=D \dualprod{d A}{\mu} (x_*) + q D V (x_*).
\end{split}
\]
\qedhere
\end{proof}

\section{Asymptotics of solutions}

\label{sectionAsymptotics}

\subsection{Upper bound on the action}
We begin the study of asymptotics by giving a sharp upper bound on the critical level \(c_\varepsilon\).

\begin{proposition}[Upper bound on the action of solutions]
\label{propositionUpperBound}
One has 
\[
 \limsup_{\varepsilon \to 0} \varepsilon^{-N} c_\varepsilon
 \le \inf_{\Lambda} \mathcal{C}.
\]
\end{proposition}

The derivation of the upper bound is based on the idea of testing the functional against rescaled test functions \citelist{\cite{BonheureVanSchaftingen2008}*{lemma~12}\cite{MorozVanSchaftingen2010}*{lemma~4.1}} 
with the phase-shift already appearing for weak magnetic fields \citelist{\cite{Cingolani2003}*{lemma~3.2}\cite{CingolaniSecchi2005}*{(27)}}.
The strong magnetic field r\'egime makes the computation more delicate.

\begin{proof}%
[Proof of proposition~\ref{propositionUpperBound}]
Given a point \(x_* \in \Lambda\) and a test function \(v_* \in C^1_c (\R^N;\C)\), we define the function \(u_{\varepsilon} : \Omega \to \R\) for every \(x \in \Omega\) by
\[
 u_{\varepsilon} (x) \defeq \e^{-\imath  A (x_*)[x - x_*]/\varepsilon^2} v_* \Bigl(\frac{x - x_*}{\varepsilon} \Bigr).
\]
Since \(x_*\) is an interior point of \(\Omega\) and the function \(v\) has compact support, for \(\varepsilon > 0\) sufficiently small, \(u_\varepsilon \in C^1_c (\Omega)\).
Moreover, for every \(x \in \Omega\),
\[
 \varepsilon D_{A/\varepsilon^2} u_{ \varepsilon} (x)
 = D_{DA(x_*)} v_* \bigl(\tfrac{x - x_*}{\varepsilon} \bigr)
 + \frac{\imath}{\varepsilon}  v_* \bigl(\tfrac{x - x_*}{\varepsilon} \bigr) \bigl(A (x) - D A (x_*)[x - x_*] - A (x_*)\bigr).
\]
Therefore,
\[
 \varepsilon^{-N} \int_{\Omega} \varepsilon^2 \abs{D_{A/\varepsilon^2} u_{\varepsilon}}^2
 = \int_{\R^N} \Bigabs{D_{DA(x_*)} v_* (y)
 + \imath v_* (y) \tfrac{A (x_* + \varepsilon y) - D A (x_*)[\varepsilon y] - A (x_*)}{\varepsilon}}^2 \dif y.
\]
Since the differential form \(A\) is differentiable at the point \(x_*\), by Lebesgue's dominated convergence theorem,
\[
 \lim_{\varepsilon \to 0} \varepsilon^{-N} \int_{\Omega} \varepsilon^2 \abs{D_{A/\varepsilon^2} u_{\varepsilon}}^2
 = \int_{\R^N} \bigabs{D_{DA(x_*)} v_*}^2.
\]

Similarly, since the potential \(V\) is continuous at \(x_*\) and since \(v_* \in L^2 (\R^N)\), 
\[
  \lim_{\varepsilon \to 0} \varepsilon^{-N} \int_{\Omega} V \abs{u_{\varepsilon}}^2
  =\int_{\R^N} V (x_*) \abs{v_*}^2.
\]
Finally, since \(x_*\) is an interior point of \(\Lambda\), if \(\varepsilon > 0\) is sufficiently small, \(\supp u_\varepsilon \subset \Lambda\) and 
\[
  \varepsilon^{-N} \int_{\Omega} G_\varepsilon (t u_{\varepsilon})
  =\frac{t^p}{p} \int_{\R^N} \abs{v_*}^p.
\]
As a consequence of the previous limits, we have that 
\[
\limsup_{\varepsilon \to 0} \sup_{t > 0} \varepsilon^{-N} \mathcal{G}_\varepsilon (t u_\varepsilon) 
\le \sup_{t > 0} \mathcal{I}_{V (x_*), D A (x_*)} (t v).
\]
Therefore,
\[
 \limsup_{\varepsilon \to 0} \varepsilon^{-N} c_\varepsilon \le \sup_{t > 0} \mathcal{I}_{V (x_*), D A (x_*)} (t v)
\]
We conclude by taking the infimum over \(v \in C^1_c (\R^N) \setminus \{0\}\) and applying lemma~\ref{lemmaCharacterizationE}.
\end{proof}

\subsection{Lower bound on the action}
We now study the asymptotic behaviour of sequences of solutions in the semiclassical r\'egime and establish a lower bound on the functional.

\begin{proposition}[Lower bound on the action of solutions]
\label{propositionLowerBound}
Let \((\varepsilon_n)_{n \in \N} \) be a sequence in \(\R^+\) that converges to \(0\), let \((u_{n})_{n \in \N}\) be a sequence of solutions of \(\mathcal{Q}_{\varepsilon_n}\) and let \((x^i_n)_{n \in \N} \subset \R^N\), \(1\leq i \leq M\), be sequences of points in \(\Omega\) such that \(x^i_n \to x_*^i \in \Lambda\) as \(n \to \infty\).
If for every \(i \in \{1, \dotsc, M\}\), \(V(x_*^i) > 0\) and
\begin{equation*}
 \liminf_{n\to\infty} \frac{1}{\varepsilon_n} \int_{B_{\varepsilon_n \rho} (x^i_n)} \abs{u_{n}}^p > 0\:,
\end{equation*}
and if for every \(i,j \in \{ 1, \dots, M \}\) such that \(i\neq j\),
\begin{equation*}
 \lim_{n\to\infty} \frac{\abs{x^i_n-x^j_n}}{\varepsilon_n} = +\infty\:,
\end{equation*}
then 
\begin{equation*}
 \liminf_{n\to\infty} \varepsilon_n^{-N} \mathcal{G}_{\varepsilon_n}(u_{\varepsilon_n}) \geq 
\sum_{i=1}^M \mathcal{C}(x_*^i)\:.
\end{equation*}
\end{proposition}

\begin{proof}%[Proof of proposition~\ref{propositionLowerBound}]
Without loss of generality, we can assume that 
\[
 \liminf_{n\to\infty} \varepsilon_n^{-N} \mathcal{G}_{\varepsilon_n}(u_{\varepsilon_n})
 =\limsup_{n \to \infty} \varepsilon_n^{-N} \mathcal{G}_{\varepsilon_n}(u_{\varepsilon_n})<\infty.
\]
We define the rescaled functions \(v_n^i : (\Omega - x_n^i)/\varepsilon_n \to \C\) for every \(n \in\N\), \(i \in \{1, \dotsc, M\}\) and \(y \in (\Omega - x_n^i)/\varepsilon_n\) by
\[
v_n^i(y) \defeq \e^{\imath  A (x_n^i)[y-x_n^i]/\varepsilon_n} u_{n}(x_n^i + \varepsilon_n y). 
\]
We first observe that for every \(R > 0\) and \(n \in \N\),
\[
 \int_{B_R \cap (\Omega - x_n^i)/\varepsilon_n} V_n^i \abs{v_n^i}^2 = \frac{1}{\varepsilon_n^N} \int_{B_{\varepsilon_n R} (x_n^i) \cap \Omega} V \abs{u_n}^2
 \le \frac{1}{\varepsilon_n^N} \int_{\Omega} V \abs{u_n}^2,
\]
with the rescaled potential \(V_n^i : (\Omega - x_n^i)/\varepsilon_n \to \C\) defined for \(y \in (\Omega - x_n^i)/\varepsilon_n\) by
\[
 V_n^i (y) \defeq V(x_n + \varepsilon_n y).
\]
Since \(V\) is continuous, we deduce that for every \(R > 0\),
\[
 \limsup_{n \to \infty} \int_{B_R \cap (\Omega - x_n^i)/\varepsilon_n} V (x_*^i) \abs{v_n^i}^2 
 \le \limsup_{n \to \infty} \int_{\Omega} V \abs{u_n}^2
\]
Next we define the rescaled vector potential \(A_n^i : (\Omega - x_n^i)/\varepsilon_n \to \bigwedge^1 \R^N\) for each point \(y \in (\Omega - x_n^i)/\varepsilon_n\) by
\[
 A_n^i (y) = \frac{A (x_n^i + \varepsilon_n y) - A (x_n^i)}{\varepsilon_n},
\]
and we observe that for every \(R > 0\) and \(n \in \N\),
\[
 \int_{B_R \cap (\Omega - x_n^i)/\varepsilon_n} \abs{D_{A_n^i} v_n^i}^2 = \frac{1}{\varepsilon_n^N} \int_{B_{\varepsilon_n R} (x_n^i) \cap \Omega} \varepsilon_n^2 \abs{D_{A/\varepsilon_n^2} u_n} \le \frac{1}{\varepsilon_n^N} \int_{\Omega} \varepsilon_n^2 \abs{D_{A/\varepsilon_n^2} u_n}.
\]
By our assumption of continuous differentiability of the differential form \(A\), for every \(R > 0\), \(A_n^i - DA(x_n^i) \to D A (x_*^i)\) uniformly on \(B_R\). Therefore, the sequence \((u_n A_n)_{n \in \N}\) is bounded in \(L^2 (B_R)\). By the classical Rellich theorem, an extraction of subsequence and a diagonal argument, there exists a function \(v_*^i \in H^1_{\mathrm{loc}} (\R^N; \C)\) such that \(v_n^i \to v_*^i\) in \(L^q_{\mathrm{loc}} (\R^N)\) if \(\frac{1}{2}-\frac{1}{N} < \frac{1}{q}\) and  
\(
 D_{A_n^i} v_n^i \to D_{DA(x_*^i)} v_*^i
\)
in \(L^2 (\R^N; \Lin(\R^N;\C))\).
By lower semi-continuity of the norm, we have for every \(R>0\),
\[
  \int_{B_R} \abs{D_{D A (x_*^i)} v_*^i}^2 + V (x_*^i) \abs{v_*^i}^2
  \le \liminf_{n \to \infty} \frac{1}{\varepsilon_n^N} \int_{\Omega} \abs{D_{A/\varepsilon_n^2} u_n}^2 + V \abs{u_n}^2.
\]
By Lebesgue's monotone convergence theorem we deduce that \(v_*^i \in H^1_{D A(x_*), V (x_*)} (\R^N)\) and 
\[
  \int_{\R^N} \abs{D_{D A (x_*^i)} v_*^i}^2 + V (x_*^i) \abs{v_*^i}^2
  \le \liminf_{n \to \infty} \frac{1}{\varepsilon_n^N} \int_{\Omega} \abs{D_{A/\varepsilon_n^2} u_n}^2 + V \abs{u_n}^2.
\]
Since by assumption the domain \(\Omega\) is smooth in a neighbourhood of \(x_*^i\), there exists a set \(\Omega_*^i\) which is either a half-space or the whole space such that \(((\Omega - x_n^i)/R)_{n \in \N}\) converges to \(\Omega_*^i\) locally in Hausdorff distance, and we have \(v_*^i \in H^1_{D A(x_*), V (x_*)} (\Omega_*^i)\)

Given the functions \(\chi_n^i = \chi_{\Lambda} (x_n^i + \varepsilon_n \cdot)\), we observe that for every \(i \in \{1, \dotsc, M\}\), the sequence \((\chi_n^i)_{n \in \N}\) is bounded in \(L^\infty (\R^N)\). Since \(L^1 (\R^N)\) is separable, up to the extraction of a subsequence, there exists thus \(\chi_*^i\) in \(L^\infty (\R^N)\) such that \(\chi_{\Lambda} (x_n^i + \varepsilon_n \cdot) \to \chi_*^i\) weakly-\(\ast\) in \(L^\infty (\R^N)\).

In particular, for every \(\varphi \in C^1_c (\Omega_*^i;\C)\), for every \(n \in \N\) large enough so that \(\supp \varphi \subset (\Omega - x_n^i)/R\) we have, since \(u_n\) is a weak solution of the penalized problem \((\mathcal{Q}_{\varepsilon_n})\), 
\[
 \int_{\R^N} \scalprod{D_{A_n^i} v_n^i}{D_{D A_n^i} \varphi} + V_n^i \scalprod{v_n^i}{\varphi} = \int_{\R^N} \scalprod{g (x_n^i + \varepsilon_n \cdot, v_n^i)}{\varphi}.
\]
As the sequence \((D_{A_n^i} \varphi)_{n \in \N}\) converges to \(D_{D A(x_*^i)} \varphi\) strongly in \(L^2 (\Omega_*^i)\), we have
\[
  \int_{\R^N} \scalprod{D_{D A (x_*^i)} v_*^i}{D_{D A (x_*^i)} \varphi} + V (x_*^i) \scalprod{v_*^i}{\varphi} = \int_{\R^N} \chi_*^i \abs{v_*^i}^{p - 2} \scalprod{v_*^i}{\varphi}.
\]
Since by definition the set \(C^1_c (\R^n;\C)\) is dense in \(H^1_{D A (x_*^i), V (x_i^*)} (\Omega_*^i)\) and \(v_*^i \in H^1_{D A (x_*^i), V(x_*^i)} (\R^N)\), we have
\[
   \int_{\R^N} \abs{D A (x_*^i) v_*^i}^2 + V \abs{v_*^i}^2
 = \int_{\R^N} \chi_*^i \abs{v_*^i}^p,
\]
and thus for every \(t \ge 0\),
\[
\begin{split}
 \frac{1}{2} \int_{\R^N} \abs{D A (x_*^i) v_*^i}^2 + V \abs{v_*^i}^2
 - \frac{1}{p} \int_{\R^N} \chi_*^i \abs{v_*^i}^p
 &\ge \frac{1}{2} \int_{\R^N} \abs{D A (x_*^i) tv_*^i}^2 + V \abs{tv_*^i}^2
 - \frac{1}{p} \int_{\R^N} \chi_*^i \abs{tv_*^i}^p\\
 &\ge \frac{1}{2} \int_{\R^N} \abs{D A (x_*^i) tv_*^i}^2 + V \abs{tv_*^i}^2
 - \frac{1}{p} \int_{\R^N} \abs{tv_*^i}^p\\
 &= \mathcal{I}_{D A (x_*^i), V (x_*^i)} (t v_*^i).
 \end{split}
\]
We have thus proved that  
\[
 \frac{1}{2} \int_{\R^N} \abs{D A (x_*^i) v_*^i}^2 + V \abs{v_*^i}^2
 - \frac{1}{p} \int_{\R^N} \chi_*^i \abs{v}^p
 \ge \sup_{t > 0} \mathcal{I}_{D A (x_*^i), V (x_*^i)} (t v_*^i) \ge \mathcal{C} (x_*^i).
\]
We observe that for every \(R > 0\), by lower-semicontinuity 
\begin{multline}
\label{eqAsymptoticsSmallBalls}
 \liminf_{n \to \infty} \frac{1}{2}\int_{B_{\varepsilon R} (x_n^i)}  \bigl(\varepsilon_n^2 \abs{\nabla u_{\varepsilon_n}}^2 + V \abs{u_{\varepsilon_n}}^2 \bigr) -\int_{B_{\varepsilon R} (x_n^i)} \mathsuper{G}_{\varepsilon_n}(u_{\varepsilon_n}) \\
\ge \frac{1}{2} \int_{B_R} \abs{D A (x_*^i) v_*^i}^2 + V \abs{v_*^i}^2
 - \frac{1}{p} \int_{B_R} \chi_*^i \abs{v_*^i}^p\\
 \ge \mathcal{C} (x_*^i) -  \frac{1}{2}  \int_{\R^N \setminus B_R} \abs{D A (x_*^i) v_*^i}^2 + V \abs{v_*^i}^2.
\end{multline}

Next we proceed as in \cite{BonheureVanSchaftingen2008}*{lemma~15} by testing the equation with \(\psi_{n, R}{}^2 u\) where the function \(\psi_{n, R} : \R^N \to \R\) is defined by \(\psi_{n, R} (x)= \prod_{i = 1}^M \psi \bigl(\frac{x - x_n}{\varepsilon_n R}\bigr)\) with \(\psi \in C^\infty (\R^N)\), \(\psi = 1\) on \(\R^N \setminus B_2\) and \(\psi=0\) on \(B_{1}\).
We compute for every \(n \in \N\) and \(R > 0\) that 
\[
 \int_{\Omega} \psi_{n, R}{}^2 \abs{D_{A/\varepsilon_n^2} u_n}^2 + V \psi_{n, R}{}^2 \abs{u_n}^2
 = \int_{\Omega} \psi_{n, R}{}^2 \scalprod{g_{\varepsilon_n} (\cdot, u_n)}{u_n}
 - 2 \int_{\Omega} \varepsilon_n^2 \scalprod{u_n D \psi_{n, R}}{\psi_{n, R} D_{A/\varepsilon_n^2} u_n}
\]
Hence, by the properties $(g_4)$ and $(g_5)$, and then $(g_2)$, if for every \(i, j \in \{1, \dotsc, M\}\) such that \(i \ne j\), \(\abs{x_n^i - x_n^j} \ge 4 \varepsilon_n R\),
\begin{multline*}
\frac{1}{\varepsilon_n^N} \int_{\Omega \setminus \bigcup_{i=1}^M B_{\varepsilon_n R}(x^i_n)} \varepsilon_n^2\abs{D_{A/\varepsilon_n^2} u_{n}}^2 + V \abs{u_{n}}^2 -  G_{\varepsilon_n}(\cdot,u_{n}) \\
\ge \frac{1}{2} \int_{\Omega \setminus\bigcup_{i=1}^M B_{\varepsilon_n R}(x^i_n)}\bigl( \varepsilon_n^2 \abs{D_{A/\varepsilon_n^2} u_{n}}^2 + V \abs{u_{n}}^2 -  \scalprod{g_{\varepsilon_n}(\cdot,u_{n})}{u_n} \bigr)\\
\shoveright{\ge - \frac{1}{\varepsilon_n^N} \sum_{i=1}^M \int_{(B_{2 \varepsilon_nR} (x_n^i) \setminus B_{\varepsilon_n R} (x_n^i) ) \cap \Omega} \abs{u_n}^p 
 + 2 \varepsilon_n^2 \scalprod{u_n D \psi_{n, R} }{\psi_{n, R} D_{A/\varepsilon_n^2} u_n}}\\
 =  - \sum_{i=1}^M \int_{(B_{2 R} \setminus B_{R} ) \cap (\Omega -x_n^i)/\varepsilon} \abs{v_n^i}^p 
 + 2 \scalprod{v_n D \psi_{R} }{\psi_{n, R} D_{A_n^i} v_n^i},
\end{multline*}
with \(\psi_{R} (y)= \psi (y /R)\).
Since for every \(i, j \in \{1, \dotsc, M\}\) such that \(i \ne j\), \(\lim_{n \to \infty} \abs{x^i_n - x^j_n}/(\varepsilon_n R) = \infty\). By the weak convergence of the sequence \((D_{A_n^i} v_n^i)_{n \in \N}\) in \(L^2 (B_R)\) and the strong convergence of the sequence \((v_n^i)_{n \in \N}\) in \(L^p (B_R)\), we deduce that 
\begin{multline*}
 \liminf_{n \to \infty} \frac{1}{\varepsilon_n^N} \int_{\R^N\setminus \bigcup_{i=1}^M B_{\varepsilon_n R}(x^i_n)} \varepsilon_n^2\abs{D_{A/\varepsilon_n^2} u_{n}}^2 + V \abs{u_{n}}^2 - G_{\varepsilon_n}(\cdot,u_{n})\\
\ge
-\sum_{i = 1}^M \int_{B_{2 R} \setminus B_R} \abs{v_*^i}^p 
 + 2 \scalprod{v_*^i D \psi_R }{\psi_{R} D_{D A (x_*^i)} v_*^i}.
\end{multline*}
In view of \eqref{eqAsymptoticsSmallBalls}, we have thus proved for every \(R > 0\) that 
\begin{multline*}
 \liminf_{n \to \infty} \frac{1}{\varepsilon_n^N} \mathcal{G}_{\varepsilon_n} (u_n)
 \ge \sum_{i = 1}^M \Bigl(\mathcal{C} (x_*^i) -  \int_{\R^N \setminus B_R} \frac{1}{2} \bigl(\abs{D_{D A (x_*^i)} v_*^i}^2 + V \abs{v_*^i}^2\bigr) \\ 
 - \int_{B_{2 R} \setminus B_R} \abs{v_*^i}^p 
 + 2 \scalprod{v_*^i D \psi_R }{\psi_{R} D_{D A (x_*^i)} v_*^i}\Bigr).
\end{multline*}
The conclusion follows by taking \(R \to \infty\), since for every \(i \in \{1, \dotsc, M\}\), \(v_*^i \in H^1_{D A (x_*^i), V (x_*^i)} (\Omega_*^i) \subset L^p (\R^N; \C)\).
\end{proof}

\subsection{Asymptotics for a family of groundstates}
In this section we transform the lower bound on the action functional obtained above

\begin{proposition}[Asymptotics for a family of groundstates]
\label{propositionVanishingSmallBalls}
If \((u_{\varepsilon})_{\varepsilon>0}\) is a family of solutions to \eqref{problemPenalized} such that 
\[
 \limsup_{\varepsilon \to 0} \varepsilon^{-N}\mathcal{G}_{\varepsilon}(u_\varepsilon)\le \inf_{\Lambda} \mathcal{C},
\]
then there exists a family of points \((x_\varepsilon)_{\varepsilon > 0}\) in \(\Lambda\) such that 
\[
 \liminf_{\varepsilon \to 0} \mathcal{C} (x_\varepsilon) = \lim_{\varepsilon \to 0} \mathcal{G}_\varepsilon (u_\varepsilon) = \inf_{\Lambda} \mathcal{C}
\]
and
\[
 \lim_{\substack{\varepsilon\to 0 \\ R \to \infty}} \norm{u_{\varepsilon}}_{L^{\infty}(\Lambda \setminus 
 B(x_{\varepsilon},\varepsilon R))}   = 0.
\]
\end{proposition}

In order to deduce proposition~\ref{propositionVanishingSmallBalls} we need to translate \(L^p\) bounds into \(L^\infty\) bounds by suitable regularity estimates.

\begin{lemma}[Vanishing in the mean implies uniform vanishing on balls]
\label{lemmaEquivalenceLpLinfty}
Let \((u_n)_{n \in \N}\) be a sequence of weak solutions of the problem $(\mathcal{Q}_{\varepsilon_n})$ in \(H^1_{A/\varepsilon_n^2,V} (\Omega)\), \(\rho > 0\) and \(x_n \in \Omega\). 
If 
\[
 \lim_{n \to \infty}\frac{1}{\varepsilon_n^N} \int_{B_{\varepsilon_n \rho} (x_n) \cap \Omega} 
 \abs{u_n}^p = 0,
\]
then 
\[
 \lim_{n \to \infty} \norm{u_n}_{L^{\infty} (B_{\varepsilon_n \rho/2} (x_n) \cap \Omega)}  = 0.
\]
\end{lemma}

A fundamental tool to prove lemma~\ref{lemmaEquivalenceLpLinfty} is Kato's inequality. In order to state it we recall that the sign of a complex number \(w \in \C\) is defined by
\[
 \sign w 
 \defeq \begin{cases}
     w/\abs{w} & \text{if \(w \ne 0\)},\\
     0 & \text{if \(w = 0\)}.
   \end{cases}
\]

\begin{proposition}[Kato's inequality \cite{Kato1972}*{lemma~A}]
\label{propositionKato}
Let \(u \in L^1_\mathrm{loc} (\Omega)\). If \(\Delta_A u \in L^1_\mathrm{loc} (\Omega)\), then \(\Delta \abs{u} \in L^1_\mathrm{loc} (\Omega)\), and 
\[
 -\Delta \abs{u} \le -\scalprod{\sign (\Bar{u})}{\Delta_A u}.
\]
\end{proposition}

\begin{proof}[Proof of lemma~\ref{lemmaEquivalenceLpLinfty}]
By Kato's inequality (proposition~\ref{propositionKato}),
\[
 -\varepsilon_n^2 \Delta \abs{u_n} \le \abs{u}^{p - 1}.
\]
weakly in \(\Omega\). We define the function \(w_n : \R^N \to \R\) to be the extension by \(0\) to \(\R^N\) of \(\abs{u_n}\).
It is clear that 
\begin{equation}
\label{ineqwn}
 -\varepsilon_n^2 \Delta w_n \le w_n^{p - 1}
\end{equation}
weakly in \(\Omega\) and that \(w_n \vert_\Omega \in H^1_0 (\Omega)\). 
We claim that \(w_n\) satisfies weakly \eqref{ineqwn} in the whole space \(\R^N\).
Let \(\varphi \in C^1_c (\R^N)\). We define \(\psi \in H^1_0 (\Omega)\) to be the unique solution of the problem
\[
  \left\{
    \begin{aligned}
     -\Delta \psi & = -\Delta \varphi & &\text{in \(\Omega\)},\\
     \psi & = 0 & & \text{on \(\partial \Omega\)}.
    \end{aligned}
  \right.
\]
Equivalently, \(\psi\) is the unique minimizer of the functional 
\[
 \frac{1}{2} \int_{\Omega} \abs{D \psi}^2 - \int_{\Omega} \scalprod{D \varphi}{D \psi};
\]
that is, the function \(\psi\) is the projection of \(\varphi \vert_\Omega \in H^1 (\Omega)\) on \(H^1_0 (\Omega)\).
By the maximum principle, \(\psi \le \varphi\) and thus \(\psi \in L^\infty (\Omega)\).

Since \(w_n\vert_\Omega \in H^1_0 (\Omega)\), we have 
\[
  \int_{\Omega} \scalprod{D w_n}{D \varphi} 
  = \int_{\Omega} \scalprod{D w_n}{D \psi} 
  = \int_{\Omega} w_n^{p - 1} \psi
  \le \int_{\Omega} w_n^{p - 1} \varphi,
\]
and thus \eqref{ineqwn} is satisfied weakly in the whole space \(\R^N\).

Since \(\frac{1}{p} < \frac{1}{2} - \frac{1}{N}\) we have \(p > \frac{p - 2}{2}N\). By the next proposition \ref{propositionNonlinearEstimate} and by scaling, we conclude that
\[
 \lim_{n \to \infty} \sup_{B_{\varepsilon_n \rho/2} (x_n) \cap \Omega} \abs{u_n} =
 \lim_{n \to \infty} \sup_{B_{\varepsilon_n \rho/2} (x_n)} w_n = 0.\qedhere
\]
\end{proof}

\begin{proposition}
\label{propositionNonlinearEstimate}
Let \(p \in [2, \infty)\) and \(q \ge [1, \infty)\). If \(q > \frac{p - 2}{2} N\), then for every \(\eta > 0\) and \(\rho < R\), there exists \(\delta > 0\) such that if \(w \in L^{p - 1} (B_R)\) and 
\begin{align*}
  w  &\ge 0 
    && \text{almost everywhere in \(B_R\)}\\
  - \Delta w &\le w^{p - 1} 
    && \text{in the sense of distributions in \(B_R\)},\\
  \int_{B_R} w^q &\le \delta,
\end{align*}
then 
\[
  \norm{w}_{L^\infty (B_\rho)} \le \eta.
\]
\end{proposition}

The exponent \(q = \frac{p - 2}{2}N\) is critical in this statement. Indeed, if \(-\Delta w \le w^{p - 1}\), \(w \in L^q\) and \(w \ne 0\), then the function \(w_\lambda (x) =w (x/\lambda)/ \lambda^{\frac{2}{p - 1}} \) satisfies also the equation, \(\lim_{\lambda \to 0} \int_{B_R} w^q = 0\) and \(\lim_{\lambda \to 0} \norm{w}_{L^\infty (B_\rho)} = \infty\).

Proposition~\ref{propositionNonlinearEstimate} will follow from a linear regularity estimate.

\begin{lemma}[Regularity of distributional subsolutions to linear problems]
\label{lemmaRegularitySubsolutions}
Let \(\rho < R\), \(r \ge 1\) and \(q \in [1, \infty]\). Assume that either \(r > \frac{N}{2}\) or 
\(q < \infty\) and 
\[
 \frac{1}{q} \ge \frac{1}{r} - \frac{2}{N}.
\]
If \(f \in L^r (B_R)\) and \(w \in L^1 (B_R)\) is nonnegative and satisfies 
\begin{align*}
 -\Delta w &\le f&
  &\text{in \(B_R\)}, 
\end{align*}
in the sense of distributions, 
then \(w \in L^q (B_\rho)\), and 
\[
 \norm{w}_{L^q (B_\rho)} \le C \bigl(\norm{f}_{L^r (B_R)} + \norm{w}_{L^1 (B_R)}\bigr).
\]
\end{lemma}
\begin{proof}
When \(r > \frac{2}{N}\), this is a classical regularity result \citelist{\cite{GilbargTrudinger1998}*{theorem 8.16}\cite{Trudinger1973}}.
Otherwise the proof follows from the same iteration argument that we develop here for the sake of completeness.

We assume without loss of generality that \(q \ge r\).
We choose a mollification kernel \(\rho \in C^2_c (\R^N)\), such that \(\supp \rho \subset B_1\) and define \(w_\lambda\) by \(w_\lambda = \lambda + \rho_\lambda \ast w \in C^2 (B_{R - \lambda})\) and \(f_\lambda = \rho_\lambda \ast f\), where \(\rho_\lambda (y) = \rho (y/\lambda) / \lambda^N\). 
We observe that \(-\Delta w_\lambda \le f_\lambda\).
We choose \(\eta \in C^1_c (B_R)\) such that \(\eta = 1\) on \(B_\rho\).
If \(\supp_{\eta} \subset B_{R - \lambda}\), we compute that 
\begin{multline*}
  (\theta - 1) \int_{B_R} \abs{D (\eta  w_\lambda^{\frac{\theta}{2}})}^2\\
  = \bigl(\tfrac{\theta}{2}\bigr)^2 \int_{B_R} \scalprod{D w_\lambda}{D (\eta^2 w_\lambda^{\theta - 1})}
    + (\theta - 2) \int_{B_R} \scalprod{D (\eta w_\lambda^{\frac{\theta}{2}})}{w_\lambda^{\frac{\theta}{2}} D \eta}
    + \int_{B_R} \abs{D \eta}^2 w_\lambda^{\theta}\\
  \le \bigl(\tfrac{\theta}{2}\bigr)^2 \int_{B_R} \scalprod{D w_\lambda}{D (\eta^2 w_\lambda^{\theta - 1})}
    + \bigabs{\tfrac{\theta}{2} - 1} \int_{B_R} \abs{D (\eta w_\lambda^{\frac{\theta}{2}})}^2
    + \bigl(1 + \bigabs{\tfrac{\theta}{2} - 1}\bigr) \int_{B_R} \abs{D \eta}^2 w_\lambda^{\theta},
\end{multline*}
and therefore, by the inequation satisfied by \(w_\lambda\),
\[
  \begin{split}
    \int_{B_R} \abs{D (\eta w_\lambda^{\frac{\theta}{2}})}^2 
      &\le C_1 \Bigl(\int_{B_R} \scalprod{D w_\lambda}{D (\eta^2 w_\lambda^{\theta})}
	+ \int_{B_R} \abs{D \eta}^2 w_\lambda^{\theta}\Bigr)\\
      &\le C_1 \Bigl(\frac{\theta}{2} \int_{B_R} f_\lambda \eta^2 w_\lambda^{\theta - 1}
	+ \int_{B_R} \abs{D \eta}^2 w_\lambda^{\theta}\Bigr).
   \end{split}
\]
If we assume now that \(\frac{1}{r} = 1 - \frac{\theta - 1}{q}\), if \(\frac{\theta}{q} \ge 1-\frac{2}{N}\), we have by the Sobolev and the H\"older inequalities,
\[
 \Bigl(\int_{B_R} (\eta^\frac{2}{\theta} w_\lambda)^q \Bigr)^{\frac{\theta}{q}}
 \le C_2 \biggl(\Bigl(\int_{B_R} (\eta^{\frac{2}{\theta}} f_\lambda)^r \Bigr)^\frac{1}{r}
	\Bigl(\int_{B_R} (\eta^\frac{2}{\theta} w_\lambda)^q \Bigr)^\frac{\theta - 1}{q} + \int_{B_R} \abs{D \eta}^2 w_\lambda^\theta \biggr)
\]
From this we deduce that 
\[
\Bigl(\int_{B_\rho} w_\lambda^{q} \Bigr)^{\frac{1}{q}}
 \le C_2 \biggl(\Bigl(\int_{B_R} f_\lambda^r \Bigr)^\frac{1}{r}
+ \Bigl(\int_{B_R} w_\lambda^{\theta} \Bigr)^\frac{1}{\theta}\biggr).
\]
If \(w_\lambda \in L^\theta (B_R)\), we deduce by Fatou's lemma that \(w \in L^{q} (B_\rho)\), and 
\[
\Bigl(\int_{B_\rho} w^{q} \Bigr)^{\frac{1}{q}}
 \le C_2 \biggl(\Bigl(\int_{B_R} f^r \Bigr)^\frac{1}{r}
+ \Bigl(\int_{B_R} \abs{D \eta}^2 w^\theta \Bigr)^\frac{1}{\theta}\biggr).
\]
The exponent \(\theta\) can be replaced by \(1\) by iterating the estimate a finite number of times and noting that \(f \in L^\theta (B_R)\) by H\"older's inequality.
\end{proof}

\begin{proof}[Proof of proposition~\ref{propositionNonlinearEstimate}]
First, the proposition follows immediately from lemma~\ref{lemmaRegularitySubsolutions} with \(f = u^{p - 1}\) when \(q > \frac{p - 1}{2} N \).
Assume now that the proposition is proved for \(q > \underline{q}\).
We have just observed that this is the case with \(\underline{q} = \frac{p - 1}{2} N\).
Observe that for every \(q > \Bar{q}\), by lemma~\ref{lemmaRegularitySubsolutions},
\[
  \norm{u}_{L^q} \le C( \norm{u}^{p - 1}_{L^{\Tilde{q}}} + \norm{u}_{L^2}),
\]
where \(\frac{1}{\Tilde{q}} = \frac{p - 1}{q} - \frac{2}{N}\).
Hence, we have proved the proposition for \(q \ge 2\) such that 
\[
 \frac{1}{q} >  \frac{p - 1}{\underline{q}} - \frac{2}{N}.
\]
By iterating this procedure, the proposition is proved for \(q > \frac{p - 2}{2} N\) and \(q \ge 2\).
\end{proof}

\begin{proof}[Proof of proposition~\ref{propositionVanishingSmallBalls}]
First, we observe that, by lemma~\ref{lemmaPositivityMagnetic},
\[
\begin{split}
\int_{\Omega} \varepsilon^2 \abs{D_{A/\varepsilon^2} u}^2 + V \abs{u_\varepsilon}^2
 &= \int_{\Omega} \scalprod{\mathsuper{g}_\varepsilon (u_\varepsilon)}{u_\varepsilon}\\
 &\le \norm{\abs{u_\varepsilon}^{p - 2}/V}_{L^\infty (\Lambda)} \int_{\Lambda} V \abs{u_\varepsilon}^2 + \mu \int_{\Omega} \varepsilon^2 H\abs{u_\varepsilon}^2\\
&\le \bigl(\norm{\abs{u_\varepsilon}^{p - 2}/V}_{L^\infty (\Lambda)} + \mu \bigr) \int_{\Omega} \varepsilon^2 \abs{D_{A/\varepsilon^2} u_\varepsilon}^2 + V \abs{u_\varepsilon}^2.
\end{split}
\]
Hence,
\[
 \Bigl(1 - \mu - \norm{\abs{u_\varepsilon}^{p - 2}/V}_{L^\infty (\Lambda)}\Bigr) \int_{\R^N} \abs{D_{A/\varepsilon^2} u}^2 + V \abs{u_\varepsilon}^2 \le 0.
\]
Since \(\inf_{\Lambda} V > 0\), this implies that 
\[
 \liminf_{n \to \infty} \norm{u_\varepsilon}_{\Lambda} > 0.
\]
Hence, there exists a family \((x_\varepsilon)_{\varepsilon >0}\) such that for every \(\rho > 0\),
\[
 \liminf_{\varepsilon \to 0} \norm{u}_{L^\infty (B_{\varepsilon \rho} (x_\varepsilon) \cap \Omega)} > 0.
\]
By lemma~\ref{lemmaEquivalenceLpLinfty}, we have 
\[
 \liminf_{\varepsilon \to 0} \varepsilon^{-N} \int_{B_{\varepsilon \rho} (x_\varepsilon) \cap \Omega} \abs{u_\varepsilon}^p > 0.
\]
By the asymptotics of proposition~\ref{propositionLowerBound} and by the upper semicontinuity of action of the limiting problem \(\mathcal{E}\) (proposition~\ref{propositionEcontinuous})
\[
 \liminf_{\varepsilon \to 0} \bigl(\varepsilon^{-N} \mathcal{G}_{\varepsilon} (u_{\varepsilon})- \inf_{\Lambda} \mathcal{C} \bigr)
 \ge \liminf_{\varepsilon \to 0} \bigl(\varepsilon^{-N} \mathcal{G}_{\varepsilon} (u_{\varepsilon})- \mathcal{C} (x_{\varepsilon})\bigr)\ge 0.
\]
By our assumption, we have 
\[
  \limsup_{\varepsilon \to 0} \bigl(\varepsilon^{-N} \mathcal{G}_{\varepsilon} (u_{\varepsilon})- \mathcal{C} (x_{\varepsilon})\bigr) 
  \le \limsup_{\varepsilon \to 0} \bigl(\varepsilon^{-N} \mathcal{G}_{\varepsilon} (u_{\varepsilon})- \inf_{\Lambda} \mathcal{C}\bigr) \le 0.
\]
Therefore we conclude that 
\[
 \lim_{\varepsilon \to 0} \varepsilon^{-N} \mathcal{G}_{\varepsilon} (u_{\varepsilon})
 = \lim_{\varepsilon \to 0} \mathcal{C} (x_{\varepsilon}) = \inf_{\Lambda} \mathcal{C}.
\]

Assume now by contradiction that 
\[
 \limsup_{\substack{R \to \infty\\ \varepsilon \to 0}} \norm{u_\varepsilon}_{L^\infty (\Lambda \setminus B_{\varepsilon R})} > 0.
\]
In that case, there are sequences \((\varepsilon_n)_{n\in \N}\) in \(\R^+\) and
\((y_n)_{n\in \N}\) in \(\Lambda\) such that 
\begin{align*}
 \lim_{n\to \infty} \varepsilon_n &= 0\:, & \liminf_{n\to \infty} \norm{u_{\varepsilon_n}}_{L^\infty (B_{\varepsilon_n \rho} (y_n)\cap \Omega)} &> 0 &
&\text{and} & \lim_{n\to \infty} \frac{\abs{x_{\varepsilon_n}-y_n}}{\varepsilon_n} &= +\infty\:.
\end{align*}
Up to a subsequence, since the set \(\Lambda\) is bounded we can assume that the sequences \((x_{\varepsilon_n})_{n \in \N}\) and \((y_n)_{n \in \N}\) converge in \(\Bar{\Lambda}\).
By lemma~\ref{lemmaEquivalenceLpLinfty}, we have
\begin{align*}
 \lim_{n \to \infty}\frac{1}{\varepsilon_n^N} \int_{B_{\varepsilon_n \rho} (y_n) \cap \Omega} 
 \abs{u_{\varepsilon_n}}^p& > 0.
\end{align*}
By the lower bound on the asymptotic action of solutions (proposition~\ref{propositionLowerBound}), we conclude that
\begin{equation*}
  \liminf_{n\to\infty} \varepsilon_n^{-N} \mathcal{G}_{\varepsilon_n}(u_{\varepsilon_n}) 
   \geq \liminf_{n\to\infty} \bigl( \mathcal{C}(x_{\varepsilon_n}) 
    + \mathcal{C}(y_n) \bigr) \geq 2 \inf_{\Lambda} \mathcal{C}\: ,
\end{equation*}
in contradiction with our assumption since \(\inf_{\Lambda} \mathcal{C} > 0\).
\end{proof}

\section{Proof of the global concentration theorem}
We have now all the tools to prove theorem~\ref{theoremGlobal}.

\begin{proof}[Proof of theorem~\ref{theoremGlobal}]
We take \(\Lambda \defeq \Omega\) so that the penalized and original problems coincide. Since the set \(\Omega\) is bounded, the existence of solutions to the problem follows from proposition~\ref{propositionExistencePenalized}. The asymptotics follow from proposition~\ref{propositionUpperBound} and proposition~\ref{propositionVanishingSmallBalls}, since the penalized functional \(\mathcal{G}_\varepsilon\) coincides with the original functional \(\mathcal{F}_\varepsilon\).
\end{proof}

\section{Proof of the local concentration theorem}

The last tool that we need to prove theorem~\ref{theoremLocal} is the following construction of barrier functions \citelist{\cite{BonheureDiCosmoVanSchaftingen2012}*{lemma 5.5 and \S 6}} (see also \cite{MorozVanSchaftingen2010}*{proposition 5.3 and \S 6}).

\begin{lemma}[Construction of barrier functions]
\label{lemmaBarrier}
Let \((x_{\varepsilon})_{\varepsilon > 0}\) be a family of points in \(\Lambda\) such that \(\liminf_{\varepsilon\to 0} d(x_{\varepsilon},\partial \Lambda) > 0\), let \(\mu \in (0,1)\) and let \(R>0\).
There exists \(\varepsilon_0 > 0\) and a family of functions \((W_{\varepsilon})_{0<\varepsilon<\varepsilon_0}\) in \(C^{1,1}(\R^N\setminus B(x_{\varepsilon},\varepsilon R))\)
such that, for \(\varepsilon \in (0,\varepsilon_0)\),
\begin{enumerate}[(i)]
\item \(W_{\varepsilon}\) satisfies the inequation
\[
 -\varepsilon^2 (\Delta + \mu H) W_{\varepsilon} + (1-\mu) V W_{\varepsilon} \geq 0 \hspace{0.5cm} \text{in} \ \R^N\setminus
 B(x_{\varepsilon},\varepsilon R),
\]
\item \(\nabla W_{\varepsilon} \in L^2(\R^N\setminus B(x_{\varepsilon},\varepsilon R))\),
\item \(W_{\varepsilon} = 1\) on \(\partial B(x_{\varepsilon},\varepsilon R)\),
\item there exist \(C, \lambda, \nu > 0\) such that for every \(x \in \R^N \setminus B(x_{\varepsilon},\varepsilon R)\),
\[
 W_{\varepsilon}(x) \leq C \exp \Bigl( -\frac{\lambda}{\varepsilon} \frac{\abs{x-x_{\varepsilon}}}{1+\abs{x-x_{\varepsilon}}} \Bigr) \bigl( 1+\abs{x}^2 \bigr)^{-\frac{N-2}{2}}.
\]
\end{enumerate}
If moreover 
\[
 \liminf_{\abs{x} \to \infty} V (x) \abs{x}^2 > 0,
\]
then there exist \(C, \lambda, \nu > 0\) such that for every \(x \in \R^N\setminus B(x_{\varepsilon},\varepsilon R)\),
\[
 W_{\varepsilon}(x) \leq C \exp \Bigl( -\frac{\lambda}{\varepsilon} \frac{\abs{x-x_{\varepsilon}}}{1+\abs{x-x_{\varepsilon}}} \Bigr) \bigl( 1+\abs{x}^2 \bigr)^{-\frac{\nu}{\varepsilon}}.
\]
\end{lemma}

We complete now the proof of theorem~\ref{theoremLocal}.

\begin{proof}[Proof of theorem~\ref{theoremLocal}]
The proof follows the lines of \cite{MorozVanSchaftingen2010}*{Proposition 5.4}.
Let \((u_\varepsilon)_{\varepsilon > 0}\) be a family of solutions to the penalized problem. 
By Kato's inequality (proposition~\ref{propositionKato}), we have
\[
 -\varepsilon^2 \Delta \abs{u_\varepsilon} + V \abs{u_\varepsilon}
 \le \scalprod{\sign (u_\varepsilon)}{-\varepsilon^2 \Delta_A u_\varepsilon + V u_\varepsilon}
  = \scalprod{\sign (u_\varepsilon)}{\mathsuper{g}_\varepsilon (u_\varepsilon)}.
\]
Hence, we deduce from $(g_2)$ and $(g_3)$ that 
\[
-\varepsilon^2 \Delta \abs{u_\varepsilon} + V \abs{u_\varepsilon}
 \le \chi_\Lambda \abs{u_\varepsilon}^{p - 1} + \mu \varepsilon^2 H_\varepsilon \abs{u_\varepsilon}.
\]
since \(p > 2\) and \(V\) is positive on \(\Lambda\), by proposition~\ref{propositionVanishingSmallBalls},  for \(R > 0\) sufficiently large and \(\varepsilon > 0\) sufficiently small, we have
\[
\left\{
\begin{aligned}
 -\varepsilon^2 \Delta \abs{u_\varepsilon} -\varepsilon^2 \mu H \abs{u_\varepsilon} + (1 - \mu) V \abs{u_\varepsilon} &\le 0 & & \text{in \(\Omega \setminus B_{\varepsilon R} (x_\varepsilon)\)},\\
  \abs{u_\varepsilon} &\le 1 & & \text{on \(\partial B_{\varepsilon R} \cap \Omega\)}.
\end{aligned}
\right.
\]

In view of the construction of supersolutions of lemma~\ref{lemmaBarrier} we deduce by the comparison principle that for every \(x \in \Omega \setminus B_{\varepsilon R} (x_\varepsilon)\),
\begin{equation*}
\begin{split}
 \abs{u_{\varepsilon}(x)} \leq W_\varepsilon (x) & \le C \exp \Bigl( -\frac{\lambda}{\varepsilon} \frac{\abs{x-x_{\varepsilon}}}{1+\abs{x-x_{\varepsilon}}} \Bigr) 
\bigl( 1 + \abs{x}^2 \bigr)^{-\frac{N-2}{2}} .
\end{split}
\end{equation*}
If \(p > \frac{N}{N - 2}\), then \(p - 1 > \frac{2}{N - 2}\) and we have for \(\varepsilon > 0\) sufficiently small
\begin{equation*}
 \abs{u_{\varepsilon}(x)}^{p-1} \le \mu \varepsilon^2 \frac{(N-2)^2}{4\abs{x-x_0}^2} \Bigl(\frac{\log \frac{\rho}{\rho_0}}{\log \frac{\abs{x-x_0}}{\rho_0}}
\Bigr)^{1+\beta}.
\end{equation*}
By definition of the penalized nonlinearity \(g_{\varepsilon}\), we have then 
\begin{equation*}
\mathsuper{g}_{\varepsilon} (u_{\varepsilon}) = \abs{u_{\varepsilon}}^{p - 2} u_\varepsilon \qquad \text{in \(\R^N\setminus\Lambda\)},
\end{equation*}
and therefore the function \(u_{\varepsilon}\) solves the original problem \eqref{problemMNLSE}.

The proof in the case \(\liminf_{\abs{x} \to \infty} V (x) \abs{x}^2 > 0\) is similar.
\end{proof}

\section*{Acknowldegment}

Part of this work was done while Jonathan Di Cosmo was a research fellow of the Fonds de la Recherche Scientifique--FNRS.

\end{document}